\documentclass[12pt,twoside]{amsart}
\usepackage{mathrsfs}
\usepackage{mathtools}
\usepackage{amssymb}
\usepackage{verbatim}
\usepackage{amsmath}
\usepackage{comment}
\usepackage{amsthm,thmtools,xcolor}
\usepackage[colorlinks,linkcolor=blue,citecolor=blue, pdfstartview=FitH]
{hyperref}
\usepackage{bm}
\usepackage{a4wide}
\usepackage[latin1]{inputenc}
\usepackage[T1]{fontenc}
\usepackage{times}
\usepackage{hyperref}
\usepackage{amssymb,latexsym}
\usepackage{enumerate}

\newcommand{\dbar}{\ensuremath{\overline\partial}}

\usepackage{etoolbox}
\makeatletter
\patchcmd\maketitle
  {\uppercasenonmath\shorttitle}
  {}
  {}{}
\patchcmd\maketitle
  {\@nx\MakeUppercase{\the\toks@}}
  {\the\toks@}
  {}
  {}{}
\patchcmd\@settitle
  {\uppercasenonmath\@title}
  {}
  {}{}
\patchcmd\@setauthors
  {\MakeUppercase{\authors}}
  {\authors}
  {}{}

\makeatletter
\newcommand{\sumprime}{\if@display\sideset{}{'}\sum%
            \else\sum'\fi}
\makeatother

\begin{document}

\numberwithin{equation}{section}

% define theorem environments
\newtheorem{theorem}{Theorem}[section]
\newtheorem{proposition}[theorem]{Proposition}
\newtheorem{conjecture}[theorem]{Conjecture}
\def\theconjecture{\unskip}
\newtheorem{corollary}[theorem]{Corollary}
\newtheorem{lemma}[theorem]{Lemma}
\newtheorem{observation}[theorem]{Observation}
\newtheorem{definition}{Definition}
\numberwithin{definition}{section} %\def\thedefinition{\unskip}
\newtheorem{remark}{Remark}
\def\theremark{\unskip}
\newtheorem{kl}{Key Lemma}
\def\thekl{\unskip}
\newtheorem{question}{Question}
\def\thequestion{\unskip}
\newtheorem{example}{Example}
\def\theexample{\unskip}
\newtheorem{problem}{Problem}

\address{Franz Luef and Xu Wang: Department of Mathematical Sciences, Norwegian University of Science and Technology, Trondheim, Norway}

\email{franz.luef@ntnu.no}
\email{xu.wang@ntnu.no}

\title[Gabor frame and Seshadri constant]{Gaussian Gabor frames, Seshadri constants and generalized Buser--Sarnak invariants}

 \author{Franz Luef and Xu Wang}
\date{\today}

\begin{abstract} We investigate the frame set of regular multivariate Gaussian Gabor frames using methods from K\"ahler geometry such as H\"ormander's $\dbar$-$L^2$ estimate with singular weight, Demailly's Calabi--Yau method for K\"ahler currents and a K\"ahler-variant generalization of the symplectic embedding theorem of McDuff--Polterovich for ellipsoids. Our approach is based on the well-known link between sets of interpolation for the Bargmann-Fock space and the frame set of multivariate Gaussian Gabor frames. We state sufficient conditions in terms of a certain extremal type Seshadri constant of the complex torus associated to a lattice to be a set of interpolation for the Bargmann-Fock space, and give also a condition in terms of the generalized Buser-Sarnak invariant of the lattice. In particular, we obtain an effective Gaussian Gabor frame criterion in terms of the covolume for almost all lattices, which is the first general covolume criterion in multivariate Gaussian Gabor frame theory. The recent Berndtsson--Lempert method and the Ohsawa--Takegoshi extension theorem also allow us to give explicit estimates for the frame bounds in terms of certain Robin constant. In the one-dimensional case we obtain a sharp estimate of the Robin constant using Faltings' theta metric formula for the Arakelov Green functions.
\bigskip


\smallskip


\end{abstract}

\maketitle

%\tableofcontents

\section{Introduction}

Let $\Lambda$ be a lattice in $\mathbb{R}^{2n}$ and 
\begin{equation}\label{eq:S}
\mathfrak{H}:=\{\Omega\in \mathfrak{gl}(n,\mathbb C): \Omega=\Omega^T, \  \text{\rm Im}\, 
\Omega \ \text{is positive definite} \}
\end{equation}
be the Siegel upper half-space. Fix $\Omega\in\mathfrak{H}$, a Gaussian Gabor system, denoted by $(g_\Omega,\Lambda)$, is the family of functions $\{\pi_\lambda g_\Omega\}_{\lambda\in\Lambda}$, where $\Lambda$ is a lattice in $\mathbb{R}^{2n}$ and  
$$
(\pi_\lambda g_\Omega)(t):=e^{2\pi i \xi^T t}g_\Omega(t-x), \ \ \lambda:=(\xi,x) \in \Lambda, \ \ \xi^T t:=\sum_{j=1}^n \xi_j t_j,
$$
denotes a time-frequency shift of the Gaussian $g_\Omega(t):=\overline{e^{\pi i t^T\Omega t}}$. Letting $I$ denote the identity matrix, we note that $g_{iI}(t)$ is the standard Gaussian $e^{-\pi |t|^2}$. The frame set of $g_\Omega$ is the set of all lattices $\Lambda$ in $\mathbb{R}^{2n}$ such that $\{\pi_\lambda g_\Omega\}_{\lambda\in\Lambda}$ is a frame for $L^2(\mathbb{R}^n)$. Here, by a frame, we mean there exist positive constants $A,B$ (called frame bounds) such that \[  A\|f\|^2\le\sum_{\lambda\in\Lambda} |(f, \pi_\lambda g_\Omega)|^2\le B\|f\|^2\quad\text{for}~~f\in L^2(\mathbb{R}^n),\]
where $(\cdot, \cdot)$ denotes the $L^2$-inner product. 
Using the Bargmann transform, one may reformulate the frame property for Gaussian Gabor frames in terms of sampling property of the Bargmann--Fock space in complex analysis. In the one-dimensional case there is a density criterion for the interpolation problem in the Bargmann-Fock space due to Lyubarskii and Seip-Wallst\'en, which implies a seminal result in the theory of Gaussian Gabor frames. Namely, that $\{\pi_\lambda g_{iI}\}_{\lambda\in\Lambda}$ is a Gabor frame if and only if the covolume $|\Lambda|<1$, \cite{Lyu92,Seip92,SW92}. Recent progress on the description of the frame set of a Gabor atom has been made for totally positive functions \cite{GS,GRS} and for rational functions \cite{bekuly21}. Note that all aforementioned results on frame sets for Gabor systems are for uniformly discrete point sets in the plane. The generalization to the higher-dimensional case has been one of the most intriguing problems in the study of Gaussian Gabor frames since the methods in \cite{Lyu92,Seip92,SW92} do not have natural counterparts in the theory of several complex variables. The reason being that the theory of sampling and interpolation in several complex variables \cite{MT, Lindholm, Gr11,GL,luma09} is far more intricate than in the one-dimensional case (see \cite[section 3]{PR}) and despite considerable effort not well understood. In particular, the following central problem in multivariate Gaussian Gabor frame theory is still open. 

\medskip
\noindent
\textbf{Problem A}: Is there an equivalent Gabor frame criterion for $(g_{iI}, \Lambda)$ only in terms of the covolume $|\Lambda|$ for \emph{almost all} lattices $\Lambda$ in $\mathbb{R}^{2n}$ ($n>1$)?

\medskip

We obtain the following partial result, which is a direct consequence of Proposition \ref{pr:1.1} and our H\"ormander criterion in section \ref{se:1.2}.

\begin{theorem}[First main theorem]\label{th:main-01}  Fix $\Omega \in \mathfrak{H}$, if $|\Lambda|<\frac{n!}{n^n}$ then  $(g_\Omega,\Lambda)$ is a Gabor frame for almost all $\Lambda$ in $\mathbb{R}^{2n}$. More precisely,  $(g_\Omega,\Lambda)$ is a Gabor frame if $|\Lambda|<\frac{n!}{n^n}$ and $(\Omega, \Lambda)$ is a transcendental pair (see the definition and  the remark below for explicit examples of the transcendental pairs).
\end{theorem}

\begin{definition}\label{de:tr}  Let
$
\Lambda^\circ:=\{(\eta, y)\in \mathbb R^n \times  \mathbb R^n: \xi^T y -x^T\eta \in \mathbb Z, \ \forall\  (\xi,x) \in \Lambda\}
$
denote the symplectic dual of $\Lambda$ (also known as the adjoint lattice of $\Lambda$). Put
\begin{equation}\label{eq:tr}
\Gamma_{\Omega, \Lambda^\circ}:=\{({\rm Im}\, \Omega)^{-1/2} z \in\mathbb C^n: z=\eta+\Omega y, \ \ (\eta, y)\in \Lambda^\circ\},
\end{equation} 
where $({\rm Im}\, \Omega)^{-1/2}$ denotes the unique positive definite matrix whose square equals $({\rm Im}\, \Omega)^{-1}$. We call $(\Omega, \Lambda)$ a transcendental pair if the complex torus 
$
\mathbb C^n/ \Gamma_{\Omega, \Lambda^\circ}
$
has no analytic subvariety of dimension $1\leq d <n$ (see Definition \ref{de:t} for a related notion).
\end{definition}

\medskip
\noindent
\textbf{Remark.}  \emph{From the definition we know that all $(\Omega, \Lambda)$  are  transcendental  in case $n=1$. In case $n=2$, by \cite[p. 161]{Sha} we know that $(\Omega, \Lambda)$ is a transcendental pair if $\mathbb C^2/ \Gamma_{\Omega, \Lambda^\circ}$ is biholomorphic to $\mathbb C^2/\Gamma$ for some  lattice
$$
\Gamma:= \mathbb Z(1,0)+\mathbb Z(0,1) +  \mathbb Z(a,b)+\mathbb Z(c,d), \ \ a,b,c,d\in \mathbb C,
$$
with the set $\{1, a, b, c, d, ad-bc\}$ being linearly independent over $\mathbb Z$ (for example, if 
$$(a, b, c, d)=(\pi, i \pi^2, i\pi^3, \pi^4), \ (\sqrt 2, i\sqrt 3, i\sqrt 5, \sqrt 7), etc$$ then $\Gamma$ is transcendental). In fact, the argument in \cite[page 160-163]{Sha} can also be used to prove that $(\Omega, \Lambda)$ is transcendental  for almost all lattices $\Lambda$ in $\mathbb C^n$. The proof of our first main theorem also suggests the following conjecture, which would answer our Problem A (just take $\Omega=iI$)}.

\medskip
\noindent
\textbf{Conjecture A}:  Let $(\Omega, \Lambda)$ be a transcendental pair (see Definition \ref{de:tr}). Then $(g_\Omega,\Lambda)$ is a Gabor frame for $L^2(\mathbb R^n)$ if and only if $|\Lambda|<\frac{n!}{n^n}$.

\medskip
\noindent
\textbf{Remark.}  \emph{Since all $(\Omega, \Lambda)$  are  transcendental  in case $n=1$, the above conjecture is precisely the Lyubarskii-Seip-Wallst\'en theorem \cite{Lyu92,Seip92,SW92} in the one dimensional case.  In case $n>1$, the only known necessary condition is $|\Lambda|<1$, which is a consequence of the Balian-Low type theorems (see \cite{AFK, GHO} or a complex analysis proof of $|\Lambda|\leq 1$ by Lindholm \cite{Lindholm}).}

\medskip

Our second main result is a multivariate Gaussian Gabor frame  criterion for general lattices (see the remark after Theorem A in section \ref{se:1.2} for the proof and Corollary \ref{co:tA} for applications).

\begin{theorem}[Second main theorem]\label{th:main-02}  Fix $\Omega \in \mathfrak{H}$ and a lattice $\Lambda$ in $\mathbb R^{2n}$. Assume that there exist $r>1$, $\beta=(\beta_1,\cdots, \beta_n)\in \mathbb R^n$ with
$$
\beta_1+\cdots+\beta_n=1, \ \ \beta_j >0, \ 1\leq j\leq n,
$$
and a holomorphic injection $f$ from the ellipsoid $$B_r^\beta:=\left\lbrace z\in \mathbb C^n: \pi\,\sum_{j=1}^n \beta_j |z_j|^2 <r^2 \right\rbrace
$$
to the torus $X:=\mathbb C^n/ \Gamma_{\Omega, \Lambda^\circ}$ (see \eqref{eq:tr}) such that
$$
f^*(\omega+i\partial\dbar\phi)= \frac{i}2 \sum_{j=1}^n dz_j \wedge d\bar z_j \ \text{on $B_r^\beta$}
$$
for some smooth function $\phi$ on $X$, where $\omega:=\frac{i}2 \sum_{j=1}^n dw_j \wedge d\bar w_j$ is the Euclidean K\"ahler form on $X$ (note that the holomorphic cotangent bundle of $X$ is trivial with global frame $\{dw_j\}$). Then $(g_\Omega,\Lambda)$ is a Gabor frame.
\end{theorem} 

\medskip
\noindent
\textbf{Remark.}  \emph{In case $(\Omega, \Lambda)$ is transcendental the above theorem is equivalent to our first main theorem. In order to prove this equivalence, we generalize (see Theorem A in section \ref{se:1.2}) McDuff--Polterovich's result \cite{MP} (see Theorem \ref{th:MP}) to all K\"ahler ellipsoid embeddings (see section \ref{ss:dem}).}

\medskip

In the one-dimensional case, we also obtain the following  frame bound estimates (see Theorem B in section \ref{se:1.3} for more results), which can be seen as an effective version of  \cite[Theorem 1.1]{BGL}.

\begin{theorem}[Third main theorem]\label{th:main-03}  Let $\Lambda$ be a lattice in $\mathbb R\times \mathbb R$. Suppose that the lattice
$$
\Gamma:=\{z \in\mathbb C: z=\eta+i y, \ \ (\eta, y)\in \Lambda^\circ\},
$$
in $\mathbb C$ is generated by $\{1, \tau\}$ with  ${\rm Im}\, \tau>1$.
Then for all $f\in L^2(\mathbb R)$ with $||f||=1$, we have
\begin{equation}\label{eq:frame-403}
 \frac{4\pi({\rm Im}\, \tau-1)|\eta(\tau)|^6 }{\left(\sum_{n\in\mathbb Z} e^{-\pi n^2 {\rm Im}\, \tau}\right)^2} \leq \sqrt 2\cdot \sum_{\lambda\in\Lambda} |(f, \pi_\lambda g_{iI})|^2 \leq \frac{{\rm Im}\, \tau}{1-e^{-C}}, \ \ C:=\frac{\pi}{4} \cdot \inf_{0\neq \lambda\in \Gamma} |\lambda|^2,
\end{equation}
where
$\eta(\tau):=e^{\pi i\tau/12} \,\Pi_{n=1}^\infty (1-e^{2\pi i n \tau})$ is the Dedekind eta function.
\end{theorem} 

The above three main theorems are special cases of the H\"ormander criterion (Theorem \ref{th:hor}), Theorem A and B in section \ref{se:1.2}. The whole paper is organized as follows. 

\tableofcontents

\subsection{Background}

Our starting point is the following observation, which is an extension of a well-known duality result for $\Omega=iI$ \cite{ja82,Lyu92,Seip92}.  

\begin{proposition}\label{pr:1.1}
$(g_\Omega, \Lambda)$ is a Gabor frame for $L^2(\mathbb R^n)$ if and only if  $\Gamma_{\Omega, \Lambda^\circ}$ (see \eqref{eq:tr})
is a set of interpolation for the Bargmann--Fock space $\mathcal F^2$ (see \eqref{eq:BF-0} and Definition \ref{de:inter-intro} in section \ref{se:1.2}).
\end{proposition}
 
Our first main theorem is an extension of the approach by Berndtsson--Ortega Cerd\`{a} \cite{BO}   for one-dimensional Gaussian Gabor frames to the multivariate case by utilizing the theory of H\"ormander's $L^2$-estimates for $\dbar$  in the higher-dimensional case, which has been developed during the past two decades and has received quite some attention \cite{OT, Ohsawa94, Dem-sin, Bo-dbar, Bern06, Bern09, Blocki, GZ, BL}. The new idea is to apply Demailly's mass concentration technique \cite{Dem93} (see \cite{Tosatti1} for a nice survey).

A crucial notion in the proof of our second main theorem is a  \textit{generalized extremal type Seshadri constant} (see Definition \ref{de:beta} and Theorem \ref{th:extremal} for the extremal property), which replaces the covolume of a lattice in the one-dimensional case. Note also that the Seshadri constant has a quite different flavor than the (Beurling) densities used in the discussion of Gabor frames, since it actually takes into account the volumina of all subvarieties of the complex torus (see Theorem \ref{th:converse}) and not just of the whole complex torus. 

Our third main theorem is based on a recent result of Berndtsson--Lempert \cite{BL}, which also yields explicit estimates for the multivariate Gaussian Gabor frame bounds (see Theorem \ref{th:c1} and Theorem \ref{th:c2}).

A short account of our other related results is given in section \ref{se:1.2} and section \ref{se:1.3} below (the H\"ormander criterion, Theorem A, B are among the most crucial ones).

\subsection{Part I: Interpolation in Bargmann--Fock space}\label{se:1.2} Consider the Bargmann--Fock space
\begin{equation}\label{eq:BF-0}
\mathcal F^2:=\{F\in \mathcal O(\mathbb C^n): ||F||^2:=\int_{\mathbb C^n} |F(z)|^2 e^{-\pi|z|^2} <\infty\},
\end{equation}
where $\mathcal O(\mathbb C^n)$ denotes the space of holomorphic functions on $\mathbb C^n$ and we omit the Lebesgue measure in the integral. Let $\Gamma$ be a lattice in $\mathbb C^n$.

\begin{definition}\label{de:inter-intro} We call $\Gamma$ a set of interpolation for $\mathcal F^2$ if there exists a constant $C>0$ such that for every sequence of complex numbers $a=\{a_\lambda\}_{\lambda\in\Gamma}$ with 
$\sum_{\lambda\in\Gamma} |a_\lambda|^2 e^{-\pi|\lambda|^2}=1,
$
there exists $F\in \mathcal F^2$ such that $F(\lambda)=a_\lambda$ for all $\lambda\in \Gamma$ and $||F||^2 \leq C$.
\end{definition}

Denote by $|\Gamma|$ (the covolume of $\Gamma$) the volume of the torus $
X:=\mathbb C^n/\Gamma$
with respect to the Lebesgue measure. In the one-dimensional case, Lyubarskii, Seip and Wallst\'en \cite{Lyu92, Seip92, SW92} (see \cite{OS} for the most general one-dimensional generalization) proved that

\begin{theorem}\label{th:LSW} A lattice $\Gamma$ in $\mathbb C$ is a set of interpolation for $\mathcal F^2$ if and only if $|\Gamma|>1$. 
\end{theorem}

Another proof of the "sufficient" part of the above theorem was given by  Berndtsson and Ortega Cerd\`{a} in \cite{BO} using the H\"ormander $\dbar$ theory. In order to make best use of the H\"ormander theory, we shall introduce the following notion of the \emph{H\"ormander constant}, which is an analogue of the Seshadri constant introduced by Demailly in \cite{Dem-sin}.

\begin{definition}\label{de:hor}  Let $\Gamma$ be a lattice in $\mathbb C^n$. The H\"ormander constant of $\Gamma$ is defined by
\begin{align*}
\iota_\Gamma:=  & \sup\{\gamma\geq 0:  \text{there exists a  $\Gamma$-invariant function $\psi$ on $\mathbb C^n$ such that $\psi+\pi|z|^2$ is psh and}  \\
& \text{$\psi=\gamma \log(|z_1|^{2/\beta_1}+\cdots+ |z_n|^{2/\beta_n})$ near $z=0$ for some $\beta_j>0$, $\beta_1+\cdots+\beta_n=1$}\},
\end{align*}
where "$\Gamma$-invariant" means that $\psi(z+\lambda) =\psi(z)$ for all $z\in \mathbb C^n$ and $\lambda\in \Gamma$; "psh" means subharmonic on each embedded disc. 
\end{definition}

By using of a result of Tosatti (see \cite[Theorem 4.6]{Tosatti}), we can prove the following:

\begin{proposition}\label{pr:hor} Assume that the only positive dimensional analytic subvariety of  $X:=\mathbb C^n/\Gamma$ is $X$ itself, then
\begin{equation}\label{eq:hor}
\iota_\Gamma= \frac{(n!)^{1/n}}{n}  |\Gamma|^{1/n}.
\end{equation}
\end{proposition}

The above result suggests the following definition.

\begin{definition}\label{de:t} A lattice $\Gamma$ in $\mathbb C^n$ is said to be transcendental if the only positive dimensional analytic subvariety of  $X:=\mathbb C^n/\Gamma$ is $X$ itself.
\end{definition}

By Proposition \ref{pr:1.1} and Proposition \ref{pr:hor}, we know that the following criterion implies our first main result --- Theorem \ref{th:main-01}. 

\begin{theorem}[The H\"ormader criterion]\label{th:hor}  Let $\Gamma$ be a lattice in $\mathbb C^n$. If $\iota_\Gamma>1$  then  $\Gamma$ is a set of interpolation for $\mathcal F^2$. In particular, if $\Gamma$ is transcendental and $|\Gamma|>\frac{n^n}{n!}$, then $\Gamma$ is a set of interpolation for $\mathcal F^2$.
\end{theorem}

In order to use the above criterion for non-transcendental lattices we have to investigate the H\"ormander constant in more depth. Our main idea is to introduce the following definition.

\begin{definition}\label{de:beta}
Let $(X, \omega)$ be an $n$-dimensional compact K\"ahler manifold. Fix $x\in X$ and take a holomorphic coordinate chart $z=\{z_j\}$ near $x$ such that $z(x)=0$. Put
\begin{equation}\label{eq:Tbeta0}
T_\beta:=\log(|z_1|^{2/\beta_1}+\cdots+ |z_n|^{2/\beta_n}), \ \ \ \beta_j >0, \ \beta_1+\cdots+\beta_n=1.
\end{equation}
The $\beta$-Seshadri constant of $(X, \omega)$ at $x\in X$ is defined by
\begin{align*}
\epsilon_x(\omega;\beta):=  & \sup\{\gamma\geq 0:  \text{there exists an  $\omega$-psh function $\psi$}  \\
& \text{on $X$ such that $\psi=\gamma T_\beta$ near $x$}\},
\end{align*}
where "$\omega$-psh" means that $\psi$ is upper semi continuous on $X$ with
$$
\omega+dd^c \psi \geq 0, \ d^c:=(\partial-\dbar)/(4\pi i),
$$
in the sense of currents on $X$. 
\end{definition}

\medskip

\noindent
\textbf{Remark.} \emph{In case $X=\mathbb C^n/\Gamma$ and $\omega= dd^c(\pi|z|^2)$, we know that $\epsilon_x(\omega;\beta)$ does not depend on $x\in X$.  Comparing the above definition with Definition \ref{de:hor}, we further get
\begin{equation}\label{eq:hor1}
\iota_\Gamma= \sup_{\beta_1+\cdots+\beta_n=1,\, \beta_j >0, \, 1\leq j\leq n} \epsilon_x(\omega;\beta).
\end{equation}
In case $\beta_1=\cdots=\beta_n=1/n$, we know that $n \epsilon_x(\omega;\beta)$ equals the classical Seshadri constant of Demailly (see \cite[section 4.4]{Tosatti}). A famous result of McDuff--Polterovich is a symplectic embedding formula \cite{MP} for the classical Seshadri constant. Theorem \ref{th:main-02} follows from the following generalization of McDuff--Polterovich's result to all $\beta$-Seshadri constants. }

\medskip

\noindent
\textbf{Theorem A.} \emph{Let $(X, \omega)$ be a compact K\"ahler manifold. Denote by $\mathcal K_\omega$ the space of K\"ahler metrics for the cohomology class $[\omega]$.  Fix  $x\in X$ and $\beta_j >0$, $1\leq j\leq n$ with $\beta_1+\cdots+\beta_n=1$. Then the $\beta$-Seshadri constant $\epsilon_x(\omega;\beta)$  is equal to the following $\beta$-K\"ahler width
\begin{align*}\label{eq:Tbeta}
c_x(\omega;\beta):=\sup \left\lbrace\pi r^2: B^\beta_r \xhookrightarrow{hol_x} (X, \tilde{\omega}), \ \exists\ \text{$\tilde{\omega} \in \mathcal K_\omega$}  \right\rbrace, \ \ B^\beta_r:=\left\lbrace z\in \mathbb C^n: \sum_{j=1}^n \beta_j |z_j|^2 <r^2 \right\rbrace. 
\end{align*}
where
 "$B^\beta_r \xhookrightarrow{hol_x} (X, \tilde{\omega})$" means that there exists a holomorphic injection $f: B^\beta_r \to X$
such that $
f(0)=x$ and  
$f^*(\tilde\omega)=\frac{i}2 \sum_{j=1}^n dz_j \wedge d\bar z_j$.}

\medskip
\noindent
\textbf{Remark.}  \emph{The above theorem implies Theorem \ref{th:main-02}. In fact, by the above theorem and \eqref{eq:hor1}, the assumption in Theorem \ref{th:main-02} implies that $\iota_\Gamma>1$. Hence our second main theorem follows from the H\"ormander criterion, Theorem \ref{th:hor} and Proposition \ref{pr:1.1}.}

\medskip

In case $X=\mathbb C^n/\Gamma$ and $\omega=dd^c(\pi|z|^2)$ is the Euclidean K\"ahler form,  we know that $B_r^\beta$ is included in $X$ if and only if 
$
(\gamma+ B_r^\beta) \cap B_r^\beta =\emptyset, \ \ \forall \ 0\neq \gamma\in \Gamma, 
$
which is equivalent to that
$$
r \leq \frac12 \sqrt{\inf_{0\neq z\in \Gamma} \sum_{j=1}^n\beta_j |z_j|^2}.
$$
Hence Theorem A gives
$$
\epsilon_x(\omega;\beta)= c_x(\omega;\beta) \geq \pi r^2, \  \ \forall \  r \leq \frac12 \sqrt{\inf_{0\neq z\in \Gamma} \sum_{j=1}^n\beta_j |z_j|^2}.
$$
Put
\begin{equation}
\mathcal B:=\{\beta\in\mathbb R^n: \beta_j\geq 0, \ 1\leq j\leq n, \ \beta_1+\cdots+\beta_n=1\},
\end{equation}
then \eqref{eq:hor1} gives
$$
\iota_\Gamma \geq  \frac{\pi}4\sup_{\beta \in \mathcal B} \inf_{0\neq z\in \Gamma} \sum_{j=1}^n\beta_j |z_j|^2.
$$
Apply the H\"ormander criterion above we get:

\begin{corollary}\label{co:tA} Let $\Gamma$ be a lattice in $\mathbb C^n$. If 
$
\sup_{\beta \in \mathcal B} \inf_{0\neq z\in \Gamma} \sum_{j=1}^n\beta_j |z_j|^2 > \frac 4\pi,
$
then the H\"ormander constant $\iota_\Gamma>1$ (see Definition \ref{de:hor}) and $\Gamma$ is a set of interpolation for $\mathcal F^2$.
\end{corollary}

\noindent
\textbf{Remark.} \emph{In case all $\beta_j$ are equal to $1/n$,
\begin{equation}\label{eq:BS}
n\cdot\inf_{0\neq z\in \Gamma} \sum_{j=1}^n\beta_j |z_j|^2= \inf_{0\neq z\in \Gamma} |z|^2
\end{equation}
 is known as the Buser--Sarnak invariant $m(\Gamma)$ (see \cite{BS, La}, \cite[Theorem 5.3.6]{Laz}) of $\Gamma$.  For general $\beta\in \mathcal B$ we call 
\begin{equation}
m_\beta(\Gamma):= \inf_{0\neq z\in \Gamma} \sum_{j=1}^n\beta_j |z_j|^2
\end{equation} 
the $\beta$-Buser--Sarnak invariant of $\Gamma$. In general, we have that $\sup_{\beta \in \mathcal B}m_\beta(\Gamma)>m(\Gamma)/n$. For example, if
$$
\Gamma:= \mathbb Z(A,0)+\mathbb Z( Ai, 0)+  \mathbb Z(0, B)+\mathbb Z(0, Bi), \ \ A>B>0,
$$
then a direct computation gives
$$
\sup_{\beta \in \mathcal B}m_\beta(\Gamma)=\frac{A^2B^2}{A^2+B^2} > \frac{B^2}2= \frac{m(\Gamma)}{2}.
$$}

\subsection{Part II: Gaussian Gabor frames}\label{se:1.3} In this section we shall show how to apply the preceding results on sets of interpolation in $\mathcal{F}^2$ in Gabor analysis. We use  the same symbol $\Gamma_{\Omega, \Lambda^\circ}$ in \eqref{eq:tr} to denote the underlying lattice
$$
\{(\eta, y)\in\mathbb R^n \times \mathbb R^n: \eta+i y\in \Gamma_{\Omega, \Lambda^\circ}\}
$$
in $\mathbb R^n\times \mathbb R^n$. By a direct computation, we know that the symplectic dual of $\Gamma_{\Omega, \Lambda^\circ}$ is equal to
$$
\Gamma_{\Omega, \Lambda}:=\{({\rm Im}\, \Omega)^{-1/2} (\xi+{\rm Re}\,\Omega x, {\rm Im}\, \Omega x) \in\mathbb R^n\times \mathbb R^n:  (\xi, x)\in \Lambda\}.
$$ 
Hence Proposition \ref{pr:1.1} gives the following:

\begin{corollary}\label{co:ff}  $(g_\Omega,\Lambda)$ defines a frame in $L^2(\mathbb R^n)$ if and only if $(g_{iI}, \Gamma_{\Omega, \Lambda})$ does.
\end{corollary}

\medskip

\noindent
\textbf{Remark.} \emph{Notice that $\Gamma_{\Omega, \Lambda}$ is equal to $f_\Omega(\Lambda)$, where 
$$
f_\Omega(\xi, x):=({\rm Im}\, \Omega)^{-1/2} (\xi+{\rm Re}\,\Omega x, {\rm Im}\, \Omega x)
$$
is a linear mapping preserving the standard symplectic form $
\omega:=d\xi^T\wedge dx$ 
on $\mathbb R^n\times \mathbb R^n$. In dimension one, we know that $(g_\Omega, \Lambda)$ defines a frame in $L^2(\mathbb R)$ if and only if $(g_{iI}, \Lambda)$  defines a frame in $L^2(\mathbb R)$ by the Theorem of Lyubarskii-Seip-Wallst\'en. However, the following result implies that this is not the case for multivariate Gabor systems.}

\begin{theorem}\label{th:uniformity} The Gabor system $(g_{iI}, (\mathbb Z\oplus \frac12\mathbb Z)^2)$ does not give a frame in $L^2(\mathbb R^2)$. Note that we have
\begin{itemize}
\item[(1)] There exists an $\mathbb R$-linear isomorphism $f$ of $\mathbb R^4$ preserving the standard symplectic form $d\xi^T\wedge dx$ 
on $\mathbb R^2\times \mathbb R^2$ such that $( g_{iI}, f(\mathbb Z\oplus \frac12\mathbb Z)^2)$ does give a frame in $L^2(\mathbb R^2)$;
\item[(2)] There exists $\Omega\in\mathfrak{H}$ such that $( g_\Omega, (\mathbb Z\oplus \frac12\mathbb Z)^2)$ does give a frame in $L^2(\mathbb R^2)$. Moreover, the set
$
\{\Omega\in \mathfrak H: \text{$( g_\Omega, (\mathbb Z\oplus \frac12\mathbb Z)^2)$ is not a frame}\}
$
is included in a closed analytic subset of the Siegel upper half-space $\mathfrak H$.
\end{itemize}
\end{theorem}

In the one-dimensional case, we obtain the following estimates thanks to Faltings' Green function formula \cite{Faltings} (see also section \ref{se:5.3}). 

\medskip 
\noindent
\textbf{Theorem B.} \emph{Let $\Lambda$ be a lattice in $\mathbb R\times \mathbb R$. Put
$$
\Gamma:=\{z \in\mathbb C: z=\eta+i y, \ \ (\eta, y)\in \Lambda^\circ\}, \ \ C:=\frac{\pi}{4} \cdot \inf_{0\neq \lambda\in \Gamma} |\lambda|^2.
$$
Then $|\Lambda|^{-1}=|\Gamma| \geq C$ and we have the following estimates:
\begin{itemize}
\item[(1)] If  $
C \geq 2$ then for all $f\in L^2(\mathbb R)$ with $||f||=1$, we have
\begin{equation}\label{eq:frame-1}
\frac{e}{4|\Lambda|} \leq \sqrt 2\cdot \sum_{\lambda\in\Lambda} |(f, \pi_\lambda g_{iI})|^2 \leq \frac{1}{(1-e^{-C})|\Lambda|};
\end{equation}
\item[(2)]  If  $
1<C< 2$ then for all $f\in L^2(\mathbb R)$ with $||f||=1$, we have
\begin{equation}\label{eq:frame-2}
\frac{(C-1) e}{C^2|\Lambda|} \leq \sqrt 2\cdot \sum_{\lambda\in\Lambda} |(f, \pi_\lambda g_{iI})|^2 \leq \frac{1}{(1-e^{-C})|\Lambda|};
\end{equation}
\item[(3)] The most general case. Suppose that $\Lambda$'s symplectic dual lattice $\Gamma$ in $\mathbb C$ is generated by $\{a, \tau\}$ with $a>0$ and $a\,{\rm Im}\, \tau>1$. Then for all $f\in L^2(\mathbb R)$ with $||f||=1$, we have
\begin{equation}\label{eq:frame-4}
 \frac{4\pi(a\, {\rm Im}\, \tau-1)|\eta(\tau/a)|^6 }{\left(\sum_{n\in\mathbb Z} e^{-\pi n^2 {\rm Im}\, \tau/a}\right)^2} \leq \sqrt 2\cdot \sum_{\lambda\in\Lambda} |(f, \pi_\lambda g_{iI})|^2 \leq \frac{a\,{\rm Im}\, \tau}{1-e^{-C}},
\end{equation}
where
$\eta(\tau):=e^{\pi i\tau/12} \,\Pi_{n=1}^\infty (1-e^{2\pi i n \tau})$ is the Dedekind eta function.
\end{itemize}}

\medskip
\noindent
\textbf{Remark.} \emph{With the notation in (3), we have
$$
|\Lambda|^{-1}=|\Gamma|=a\,{\rm Im}\,\tau,
$$
thus one may look at \eqref{eq:frame-4} as an effective version of  \cite[Theorem 1.1]{BGL} (or the corresponding Theorem \ref{th:LSW}). By our definition \eqref{eq:BS} of the Buser--Sarnak constant, we have
$$
C=\frac{\pi}{4} m(\Gamma).
$$
The lower bound estimate in \eqref{eq:frame-4} is based on a precise Robin constant estimate (see Theorem \ref{th:robin}). For the upper bound, note that by  \cite[Theorem 37]{Siegel}, we have
$$
C \leq \frac{\pi}{2\sqrt 3}\, |\Gamma|,
$$
where equality holds if and only if $\Gamma$ is the hexagonal lattice. Hence the best upper bound in \eqref{eq:frame-4} (for fixed $a\,{\rm Im}\,\tau$)  is attached if and only if $\Gamma$ is the hexagonal lattice. This fact is compatible with the Strohmer--Beaver conjecture (see \cite[section 2.4]{FS}). However, the upper bound in \eqref{eq:frame-4} is not optimal in general.}

\section{Preliminaries}

\subsection{Gabor transform and Bargmann transform}

For $\Omega\in\mathfrak{H}$ (see \eqref{eq:S}) we identify $\mathbb{R}^n\times\mathbb{R}^n$ with $\mathbb{C}^n$ via $(\xi,x)\mapsto z:=\xi+\Omega x$. We call the short-time Fourier transform of $f\in L^2(\mathbb{R}^n)$ with respect to $$g_{\Omega}(t)=\overline{e^{\pi i t^T\Omega t}}$$ the Gabor transform of $f$:
$$
\mathcal V_{g_\Omega}f (\xi,x)=\int_{\mathbb R^n} f(t) e^{-2\pi i \xi^T t} \overline{g_{\Omega}(t-x)} dt_1\cdots dt_n.
$$
The latter can be related to the \emph{$\Omega$-Bargmann transform}
$$
\mathcal B_\Omega  f(z):=\int_{\mathbb R^n} f(t) e^{\pi i t^T\Omega t}e^{-2\pi i z^T t} dt_1\cdots dt_n.
$$
as follows:
$$
\mathcal V_{g_\Omega}f (\xi,x)= \mathcal B_\Omega  f(z) e^{\pi i x^T\Omega x}.
$$
%$$
%e^{\pi i x^T\Omega x} \int_{\mathbb R^n} f(t) e^{\pi i t^T\Omega t}e^{-2\pi i (\xi+\Omega x)^T t} dt_1\cdots dt_n.
%$$
By Moyal's identity we have (note than $||g_\Omega||^2=(2^n \det ({\rm Im}\,\Omega))^{-1/2}$)
%Thus the energy sum identity $
%||f||^2 \cdot ||g||^2= ||\mathcal V_gf||^2 $ gives
\begin{equation}\label{eq:iso1}
 \int_{\mathbb C^n} |\mathcal B_\Omega f(z)|^2 e^{-2\pi x^T {\rm Im\,}\Omega  x} d\xi_1 \wedge dx_1 \wedge \cdots \wedge d\xi_n \wedge dx_n =(2^n \det ({\rm Im}\,\Omega))^{-1/2}\cdot ||f||^2.
\end{equation}
Since $z=\xi+\Omega x$ implies that
$$
d\xi_1 \wedge dx_1 \wedge \cdots \wedge d\xi_n \wedge dx_n = (\det ({\rm Im}\,\Omega))^{-1}  \left( \frac i2 \partial\dbar |z|^2\right)^n,
$$
we know that \eqref{eq:iso1} gives
\begin{equation}\label{eq:iso1new}
(\det ({\rm Im}\,\Omega))^{-1} ||\mathcal B_\Omega f||_\Omega^2 =(2^n \det ({\rm Im}\,\Omega))^{-1/2}\cdot ||f||^2,
\end{equation}
where
$$
||\mathcal B_\Omega f||_\Omega^2:=\int_{\mathbb C^n} |\mathcal B_\Omega f(z)|^2 e^{-2\pi x^T {\rm Im\,}\Omega  x} \left( \frac i2 \partial\dbar |z|^2\right)^n.
$$
Sometimes we shall omit the Lebesgue volume form $\left( \frac i2 \partial\dbar |z|^2\right)^n$ in the above. Now we are ready to introduce the following definition. 

\begin{definition}\label{de:phi_o} Fix $\Omega\in \mathfrak{H}$ and put (notice that $({\rm Im}\, \Omega) x= {\rm Im}\, z$)
$$
\phi_\Omega(z):=x^T {\rm Im\,}\Omega \,  x= ({\rm Im}\,z)^T({\rm Im}\, \Omega)^{-1} {\rm Im}\, z.
$$ 
We call the space of holomorphic functions $F$ on $\mathbb C^n$ with
$$
||F||^2_\Omega :=\int_{\mathbb C^n} |F(z)|^2 e^{-2\pi \phi_\Omega(z)} <\infty, 
$$
 the  $\Omega$-Bargmann--Fock space and denote it by $\mathcal F^2_\Omega$. 
\end{definition}

\noindent
\textbf{Remark.} \emph{In case $\Omega= i I_n$, where $I_n$ denotes the identity matrix, $\mathcal F^2_\Omega$ is precisely the following \emph{classical Bargmann--Fock space} (see \cite[page 7]{Mumford} for the related Von-Neumann--Stone theorem)
\begin{equation}\label{eq:bf-classical}
\mathcal B:=\left\lbrace F\in \mathcal O(\mathbb C^n): \int_{\mathbb C^n} |F(z)|^2 e^{-2\pi |{\rm Im}\, z|^2}  \, d\xi_1\cdots d\xi_n dx_1\cdots dx_n <\infty \right\rbrace.
\end{equation}
In time-frequency analysis, the Bargmann--Fock space $\mathcal F^2$  in \eqref{eq:BF-0} is more widely used. But these two spaces are naturally isomorphic to each other since
$$
|F(z)|^2 e^{-2\pi \phi_{i I_n}(z)} = |F(z) e^{-\pi z^T z/2}|^2 e^{-\pi z^T z}.
$$}

Let us briefly recall the basics of Gaussian Gabor frames: We associate to the Gabor system $(g_\Omega,\Lambda)$ the following operators: 
\begin{itemize}
    \item \emph{analysis operator} $C^\Lambda_{g_\Omega}$ is a map from $L^2(\mathbb R^n)$ to $l^2(\Lambda)$ defined by  $f\mapsto \{(f, \pi_\lambda g_\Omega)\}_{\lambda\in \Lambda}$; 
    \item \emph{synthesis operator:} $D_{g_\Omega}^\Lambda$ is a map from $l^2(\Lambda)$ to $L^2(\mathbb R^n)$ given by $D_{g_\Omega}^\Lambda c=\sum_{\lambda\in \Lambda} c_\lambda \pi_\lambda {g_\Omega}$;
    \item \emph{frame operator:} $S_{g_\Omega}^\Lambda:=D_{g_\Omega}^\Lambda \circ C^\Lambda_{{g_\Omega}}$ is an operator on $L^2(\mathbb{R}^n)$: 
    \[
    S_g^\Lambda f=\sum_{\lambda\in \Lambda} (f, \pi_\lambda g_\Omega) \pi_\lambda g_\Omega.
    \]
\end{itemize}
An elementary computation shows that $(C^\Lambda_{g_\Omega})^*=D_{g_\Omega}^\Lambda$ and thus $S_{g_\Omega}^\Lambda:=(C_{g_\Omega}^\Lambda)^\ast \circ C^\Lambda_{g_\Omega}$ is a selfadjoint operator. The following result is well known, see \cite{Gr01}.

\begin{lemma}\label{le:bdd} For an arbitrary lattice $\Lambda$ in $\mathbb R^{n} \times \mathbb R^n$, the coefficient operator $C^\Lambda_{g_\Omega}$ from $L^2(\mathbb R^n)$ to $l^2(\Lambda)$ is bounded. 
\end{lemma}

\begin{proof} We shall give a proof for readers' convenience. Using the $\Omega$-Bargmann transform, it suffices to show that 
$$
\sum_{z\in \Gamma} |F(z)|^2e^{-2\pi ({\rm Im}\,z)^T({\rm Im}\, \Omega)^{-1} {\rm Im}\, z} \leq C \, ||F||^2_\Omega, \ \ \forall \ F\in \mathcal F^2_\Omega,
$$
where $\Gamma:=\{z=\xi+\Omega x: (\xi,x)\in \Lambda\}$. Since for $w:=({\rm Im}\, \Omega)^{-1/2}z$, we have
$$
|F(z)|^2 e^{-2\pi ({\rm Im}\,z)^T({\rm Im}\, \Omega)^{-1} {\rm Im}\, z}=|F(({\rm Im}\, \Omega)^{1/2}w) e^{\frac\pi 2 w^T w}|^2 e^{-\pi |w|^2}, \
$$
we know the above inequality is equivalent to that
$$
\sum_{w\in \Gamma'} |F(w)|^2e^{-\pi|w|^2} \leq C \, ||F||^2, \ \ \forall \ F\in \mathcal F^2,
$$
where $\Gamma':=\{w=({\rm Im}\, \Omega)^{-1/2}(\xi+\Omega x): (\xi,x)\in \Lambda\}$. By the submean inequality, we know that the above inequality is true for
$$
C=\int_{|w|<R} e^{-\pi|w|^2}, \ \ R:=\inf_{w\in \Gamma', \, w\neq 0} |w|/2,
$$
hence the lemma follows.
\end{proof}

There is a fundamental duality theory (see \cite{DLL, Janssen95, RS, CKL} and \cite[Theorem 4.22]{JL}) that links the Gabor system $(g_\Omega,\Lambda)$ with another Gabor system associated to the symplectic dual lattice/adjoint lattice defined in Definition \ref{de:tr}. 

\begin{theorem}[Duality Theorem]\label{th:dual-franz}
  $(g_\Omega,\Lambda)$ is a Gabor frame for $L^2(\mathbb{R}^n)$ with bounds $A$ and $B$ if and only if $\{\pi_{\lambda^\circ}g_\Omega\}_{\lambda^\circ\in\Lambda^\circ}$ is a Riesz sequence with bounds $A|\Lambda|$ and $B|\Lambda|$,  i.e. we have 
  \[A|\Lambda|\,\|c\|^2\le\|\sum_{\lambda^\circ \in\Lambda^\circ}c_{\lambda^\circ}\pi_{\lambda^\circ}g_\Omega\|^2 \le B|\Lambda|\,\|c\|^2\]
  for all $c\in\ell^2(\Lambda^\circ)$.
\end{theorem}
There is an intricate link between Gabor analysis and Bargmann-Fock spaces: a Gaussian Gabor system $(g_\Omega,\Lambda)$ is a frame if and only if $\Lambda$ is a set of sampling for $\mathcal{F}^2_\Omega$, and $(g_\Omega,\Lambda^\circ)$ is a Riesz basis for its closed linear span if and only if $\Lambda^\circ$ is a set of interpolation for $\mathcal{F}^2_\Omega$, see \cite{GL} for the standard case $\Omega=iI$. 

%\begin{definition}\label{de:synthesis-frame} We call the adjoint of coefficient operator  the synthesis operator, and denote it by . Moreover, the frame operator is defined by $S_g^\Lambda:=D_g^\Lambda \circ C^\Lambda_{g}$.
%\end{definition}

%\textbf{Remark}: By the above definition, we have
%$$
%D_g^\Lambda c=\sum_{\lambda\in \Lambda} c_\lambda \pi_\lambda g, \ \  \ \  .
%$$

%\begin{definition} We call the following linear subspace in $L^2(\mathbb R^n)$
%$$
%M^1(\mathbb R^n):=\{f\in L^2(\mathbb R^n): \int_{\mathbb C^n} |\bm{\mathcal B} f(z)| e^{-\pi |{\rm Im}\, z|^2}  \, d\xi_1\cdots d\xi_n dx_1\cdots dx_n <\infty\}
%$$
%the modulation space (also known as the Feichtinger algebra, see \cite{F06} for the background). 
%\end{definition}

%\textbf{Remark}: The Gabor transform with respect to a general complex valued function $g\in L^2(\mathbb R^n)$ is defined by
%$$
%\mathcal V_g f (\lambda):=(f, \pi_\lambda g),
%$$
%where
%$$
%(\pi_\lambda g)(t):=e^{2\pi i \xi\cdot x}g(t-x).
%$$
%It is known that (see \cite[Proposition 12.1.1]{Gr01}) 
%$$
%M^1(\mathbb R^n)=\{f\in L^2(\mathbb R^n): V_ff\in L^1(\mathbb R^{n} \times \mathbb R^n)\}.
%$$
%Moreover, the following lemma (see \cite[Theorem 12.2.3]{Gr01}) holds:
Motivated by this we introduce the following well known notions: 

\begin{definition}\label{de:sam-int} Let $T: H_1 \to H_2$ be a bounded $\mathbb C$-linear map between two complex Hilbert spaces. Then
\begin{itemize}
\item[(1)] $T$ is called \emph{sampling} if there exist constants $A, B>0$ such that
\begin{equation}\label{eq:samp-c}
A\,||f||^2 \leq ||Tf||^2 \leq B\,||f||^2, \ \ \forall \  f\in H_1;
\end{equation}
\item[(2)] $T$ is referred to as \emph{interpolating} if $T$ is surjective and there exist constants $A, B>0$ such that
\begin{equation}\label{eq:inte-c}
A\,||f_c||^2 \leq ||c||^2 \leq B\, ||f_c||^2, \ \ \ \forall \ c\in H_2,
\end{equation}
where $f_c$ denotes the (unique) solution of $T(\cdot)=c$ with minimal norm.
\end{itemize}
The constants $A, B$ above are called the sampling (interpolating) bounds.
\end{definition}

\begin{proposition}\label{pr:dual1} $T$ is sampling with \eqref{eq:samp-c} if and only if $T^*$ is interpolating with \eqref{eq:inte-c}.
\end{proposition}

\begin{proof} Assume that $T$ is sampling with \eqref{eq:samp-c}. Then the eigenvalues of $T^*T$ lie in $[A, B]$, thus $T^*T$ has an inverse, say $S:=(T^*T)^{-1}$, which implies that 
$T^* T S f=f, \ \ \forall \ c\in H_1$. 
Thus $T^*: H_2\to H_1$ is surjective and the minimal solution of $T^*(\cdot)=f$ is $TSf$ (note that $TSf$ is minimal since $TSf \bot \ker \, T^*$). We need to show that
$$
A||TSf||^2 \leq ||f||^2 \leq B |||TSf||^2.
$$
In fact,  $S^{-1}\geq A\,I$ implies that
$$
||TSf||^2=(TSf, TSf)=(Sf, T^*TSf)=(Sf, f)\leq \frac1A ||f||^2,
$$
thus $A||TSf||^2 \leq ||f||^2$. Moreover, $f=T^*TSf$ implies
\begin{align*}
||f||=\sup_{||g||=1} (T^*TSf, g) & =\sup_{||g||=1} (TSf, Tg)  \\
& \leq ||TSf|| \cdot \sup_{||g||=1} ||Tg|| \leq \sqrt{B} ||TSf||,
\end{align*}
which gives $||f||^2 \leq B ||TSf||^2$. This establishes one direction and the other direction may be deduced in a similar manner.
\end{proof}

The following theorem follows directly from Theorem \ref{th:dual-franz} and Definition \ref{de:sam-int}.

\begin{theorem}\label{th:dual2} Let $\Lambda$ be a lattice in $\mathbb R^{n} \times \mathbb R^n$. Then $C_{g_\Omega}^\Lambda$ is sampling with 
$$
A\, I \leq S_{g_\Omega}^\Lambda\leq B\,I
$$ 
if and only if $D_{g_\Omega}^{\Lambda^\circ}$ is sampling with 
$$
|\Lambda|\cdot A \leq (S_{g_\Omega}^{\Lambda^\circ})^* \leq |\Lambda|\cdot B.
$$
\end{theorem}

Notice that $C_{g_\Omega}^\Lambda$ is sampling if and only $(g_\Omega,\Lambda)$ defines a frame in $L^2(\mathbb R^n)$. The density theorem for Gabor frames states that if $C_{g_\Omega}^\Lambda$ is sampling, then $|\Lambda| \leq 1$. Furthermore, a \emph{Balian-Low type theorem} (see \cite[Theorem 1.5]{AFK} or \cite{GHO} for related results associated to general Fock spaces) further gives: 

\begin{theorem}\label{th:sam1} Given a lattice $\Lambda$ in $\mathbb R^{n} \times \mathbb R^n$. If $C_{g_\Omega}^\Lambda$ is sampling then $|\Lambda|<1$.
\end{theorem}

The above two theorems and Proposition \ref{pr:dual1} imply

\begin{corollary}\label{co:dual}  Given a lattice $\Lambda$ in $\mathbb R^{n} \times \mathbb R^n$.  Then $C_{g_\Omega}^\Lambda$ is sampling if and only if $C_{g_\Omega}^{\Lambda^\circ}$ is interpolation. Moreover, the interpolation bounds are a scalar multiple of  the sampling bounds. In particular, $C_{g_\Omega}^\Lambda$ can not be both sampling and interpolation. 
\end{corollary}

\begin{proof} The first part follows directly from Theorem \ref{th:dual2} and Proposition \ref{pr:dual1}. For the second part, notice that if $C_{g_\Omega}^\Lambda$ is both sampling and interpolation, we must have
$$
|\Lambda|<1, \ \ |\Lambda^\circ|<1,
$$
which is a contradiction since $|\Lambda| \cdot |\Lambda^\circ|=1$.
\end{proof}

\noindent
\textbf{Remark.}  \emph{In case $n=1$ and 
$$
g(t)=\overline{e^{\pi i a t^2}}, 
$$
for some complex number $a$ with ${\rm Im} \,a>0$, we known that $C_{g}^\Lambda$ is sampling \emph{if and only if} $|\Lambda| <1$ (see \cite{Lyu92, Seip92, SW92}). For general $n$, $g_\Omega(t):= \overline{e^{\pi i t^T\Omega t}}$, Theorem \ref{th:uniformity} implies that there exists a lattice $\Lambda$ in $\mathbb R^{n} \times \mathbb R^n$ such that $C_{g_{\Omega}}^\Lambda$ is sampling (resp. interpolation) for some $\Omega\in \mathfrak{H}$ but not  for all $\Omega\in \mathfrak{H}$. On the other hand, if $C_{g_{\Omega_0}}^\Lambda$ is sampling for some $\Omega_0\in \mathfrak{H}$ then by Theorem 1.3 in \cite{AFK}, we know that $C_{g_\Omega}^\Lambda$ is sampling if $\Omega$ is very close to $\Omega_0$.}

\subsection{Proof of Proposition \ref{pr:1.1}} 

\begin{proof}[Proof of Proposition \ref{pr:1.1}] By our definition, $(g_\Omega,\Lambda)$ defines a frame in $L^2(\mathbb R^n)$ if and only if $C_{g_\Omega}^\Lambda$ is sampling, which is equivalent to that $C_{g_\Omega}^{\Lambda^\circ}$ is interpolation (see Corollary \ref{co:dual}). Using the $\Omega$ Bargmann transform,  we know that $C_{g_\Omega}^{\Lambda^\circ}$ is interpolation if and only if $C_{g_\Omega}^{\Lambda^\circ}$ is bounded  and 
$$
\Gamma:= \{\xi+\Omega x\in \mathbb C^n: (\xi, x)\in \Lambda^\circ\}
$$
is a set of interpolation for $\mathcal F^2_\Omega$. By Lemma \ref{le:bdd}, we know that $C_{g_\Omega}^{\Lambda^\circ}$ is always bounded, hence $(g_\Omega,\Lambda)$ defines a frame in $L^2(\mathbb R^n)$ if and only if $\Gamma$ is a set of interpolation for $\mathcal F^2_\Omega$.  Notice that
$$
|F(z)|^2 e^{-2\pi \phi_\Omega(z)}=|F(({\rm Im}\, \Omega)^{1/2}w) e^{\frac\pi 2 w^T w}|^2 e^{-\pi |w|^2}, \ \  w:=({\rm Im}\, \Omega)^{-1/2}z 
$$
implies that
\begin{equation}\label{eq:iso}
F(z) \mapsto F(({\rm Im}\, \Omega)^{1/2}w) e^{\frac\pi 2 w^T w}
\end{equation}
defines an isomorphism from  $\mathcal F^2_\Omega$ to $\mathcal F^2$. Thus $(g_\Omega,\Lambda)$ defines a frame in $L^2(\mathbb R^n)$ if and only if 
$$
({\rm Im}\, \Omega)^{-1/2}\Gamma = \Gamma_{\Omega, \Lambda^\circ}
$$ 
is a set of interpolation for $\mathcal F^2$.
\end{proof}

 The duality principle Theorem \ref{th:dual2} further implies:

\begin{theorem}\label{th:dual3}  With the notation in the above proof, the following statements are equivalent:

\begin{itemize}
\item[(1)] $\Gamma_{\Omega, \Lambda^\circ}
$ is a set of interpolation for $\mathcal F^2$ and for all $F\in \mathcal F^2$ with $$\sum_{\gamma\in \Gamma_{\Omega, \Lambda^\circ}} |F(\gamma)|^2e^{-\pi|\gamma|^2} =1,$$ we have
$$
A\leq \inf_{F' \in \mathcal F^2,\, F'=F \, \text{on} \,\Gamma_{\Omega, \Lambda^\circ} } ||F'||^2 \leq B;
$$ 
\item[(2)] $\Gamma
$ is a set of interpolation for $\mathcal F_\Omega^2$ and for all $F\in \mathcal F_\Omega^2$ with 
$$\sum_{\gamma\in \Gamma} |F(\gamma)|^2e^{-2\pi \phi_\Omega(\gamma)} =1,$$
we have 
$$
A \cdot \det({\rm Im}\,\Omega)\leq \inf_{F' \in \mathcal F_\Omega^2,\, F'=F \, \text{on} \,\Gamma} ||F'||^2 \leq B \cdot \det({\rm Im}\,\Omega);
$$ 
\item[(3)]  $(\Lambda, g_\Omega)$ defines a frame in $L^2(\mathbb R^n)$  and for  all $f\in L^2(\mathbb R^n)$, $||f||=1$, 
$$
\frac{(B\cdot |\Lambda|)^{-1}}{\sqrt{2^n\det ({\rm Im}\,\Omega)}}   \leq \sum_{\lambda\in\Lambda} |(f, \pi_\lambda g_\Omega)|^2 \leq \frac{(A\cdot |\Lambda|)^{-1}}{\sqrt{2^n\det ({\rm Im}\,\Omega)}}  .
$$
\end{itemize}
\end{theorem}

\begin{proof} \eqref{eq:iso} implies $(1)\Leftrightarrow (2)$.  By \eqref{eq:iso1new}, we know that $(2)$ is equivalent to that $C_{g_\Omega}^{\Lambda^\circ}$ is interpolation (see \eqref{eq:inte-c}) with 
$$
\frac{B^{-1}\cdot ||f_c||^2}{\sqrt{2^n\det ({\rm Im}\,\Omega)}} \leq ||c||^2 \leq \frac{A^{-1}\cdot ||f_c||^2}{\sqrt{2^n\det ({\rm Im}\,\Omega)}}, \ \ \ \forall \ c\in l^2.
$$
By Proposition \ref{pr:dual1}, the above inequality is equivalent to that
$$
\frac{B^{-1}}{\sqrt{2^n\det ({\rm Im}\,\Omega)}}\,I \leq (S_{g_\Omega}^{\Lambda^\circ})^* \leq \frac{A^{-1}}{\sqrt{2^n\det ({\rm Im}\,\Omega)}}\,I.
$$
Thus Theorem \ref{th:dual2} gives $(2)\Leftrightarrow (3)$. 
\end{proof}

\subsection{Proof of the H\"ormander criterion (Theorem \ref{th:hor})}

\subsubsection{$L^2$-estimate for the $\dbar$-equation} We shall use the following special case of  H\"ormander's theorem (see \cite[page 378, Theorem 6.5]{Demailly12}, see also Chapter 4 in \cite{H65}):

\begin{theorem}\label{th:hor1} Fix a smooth $(0,1)$-form $v$ with $\dbar v=0$ on $\mathbb C^n$. Let $\phi$ be a plurisubharmonic function such that $\phi-\delta |z|^2$ is also plurisubhamonic on $\mathbb C^n$ for some positive constant $\delta$.  Then there is a smooth function $a$ on $\mathbb C^n$ such that $\dbar u=v$ and
$$
\int_{\mathbb C^n} |u|^2 e^{-\phi} \leq  \frac1\delta\int_{\mathbb C^n} |v|^2 \, e^{-\phi},
$$
where $|v|^2:=\sum |v_{\bar j}|^2$ for $v=\sum v_{\bar j} d\bar z_j$.
\end{theorem}

\begin{proof}[Proof of the H\"ormander criterion (Theorem \ref{th:hor})]
Notice that the $\beta$-Seshadri constant does not depend on the choose of $x\in \mathbb C^n/\Gamma$. Thus, if  the H\"ormander constant is bigger than one then there exist $\gamma>1$ and an  $\omega_{\rm euc}$-psh function $\psi$ on $\mathbb C^n/\Gamma$ such that $\psi=\gamma T_\beta$ near $0\in \mathbb C^n/\Gamma$  for some $\beta\in \mathcal B$. Let
$$
p: \mathbb C^n \to \mathbb C^n/\Gamma,
$$
be the natural quotient mapping. Fix $c=\{c_\lambda\}$ such that
$$
\sum_{\lambda\in \Gamma} |c_\lambda|^2 e^{-\pi|\lambda|^2}=1.
$$ 
Let us apply Theorem \ref{th:hor1} to 
$$
\phi(z):=\pi|z|^2+\frac{\psi(p(z))}{\gamma}, \ \ \ v(z):= \sum_{\lambda\in\Gamma}  c_\lambda e^{\pi\bar\lambda(z-\lambda)} \dbar \chi(|z-\lambda|),
$$ 
where $\chi$ is a smooth function on $\mathbb R$ that is equal to $1$ near the origin and equals to $0$ outside a smooth ball of radius $r$. Let us take $r$ such that
$$
\{z\in\mathbb C^n: |z-\lambda|<r\} \cap \{z\in\mathbb C^n: |z-\lambda'|<r\}=\emptyset, \ \ \forall \ \lambda\neq \lambda' \in \Gamma.
$$
Then we know that $v$ is smooth, $\dbar v=0$ and 
$$
\int_{\mathbb C^n} |v|^2 e^{-\phi} < C,
$$
for some constant thay does not depend on the sequence $c=\{c_\lambda\}$. Moreover, since $\psi$ is $\omega_{\rm euc}$-psh, we know that $\phi(z)-(1-\gamma^{-1})\pi |z|^2$ is plurisubharmonic. Thus Theorem \ref{th:hor1} implies that there exists a smooth function $u$ such that $\dbar u=v$ and 
\begin{equation}\label{eq:u}
\int_{\mathbb C^n} |u|^2 e^{-\phi}<\frac{C}{(1-\gamma^{-1})\pi}.
\end{equation}  
By a direct computation we know that  $e^{-T_\beta}$ is not integrable near $0\in \mathbb C^n/\Gamma$, hence $e^{-\phi}$ is not integrable near $\Gamma$ and \eqref{eq:u} implies that $u$ vanishes at $\Gamma$. Take
$$
F(z)= \sum_{\lambda\in\Gamma}  c_\lambda e^{\pi\bar\lambda(z-\lambda)}  \chi(|z-\lambda|) - u(z),
$$
we know that $F$ is holomorphic in $\mathbb C^n$,
$$
\int_{\mathbb C^n} |F(z)|^2 e^{-\pi|z|^2} <C_1,
$$
for some constant $C_1$ does not depend on $c$ (notice that $\psi$ is bounded from above) and $F(\lambda)=c_\lambda$ for all $\lambda\in\Gamma$. Thus $\Gamma$ is a set of interpolation. The final statement is a direct consequence pf Proposition \ref{pr:hor}, which will be proved in section \ref{ss:dem}.
\end{proof}

In order to estimate the $L^2$ norm of the extension $F$ in the above proof, we shall introduce the following Ohsawa--Takegoshi type theorem \cite{OT} proved by Berndtsson and Lempert (see \cite[Theorem 3.8]{BL}, the main theorem in \cite{GZ} and \cite{Blocki} for related results). 

\begin{theorem}\label{th:OT-xu} Let $\Gamma$ be a lattice in $\mathbb C^n$. Assume that there exists a non positive $\Gamma$ invariant function $\psi$ on $\mathbb C^n$ such that $\psi(z)+\pi|z|^2$ is plurisubharmonic on $\mathbb C^n$, $\psi$ is smooth outside $\Gamma$ and $
\psi(z) - \gamma \log |z|^2$
is bounded near the origin for some constant $\gamma>n$. Then for every sequence of complex numbers $\{c_{\lambda}\}_{\lambda\in \Gamma}$ with 
$
\sum_{\lambda\in \Gamma} |c_\lambda|^2 e^{-\pi |\lambda|^2}=1,
$
there exists $F\in \mathcal F^2$ such that $F(\lambda)=c_\lambda$ for all $\lambda\in \Gamma$ and
$$
||F||^2  \leq  \left(1-\frac n\gamma\right)^{-1}  \cdot  \frac{\pi^n}{n!} \cdot e^{-\frac n\gamma \rho}, \ \  \ \rho:= \liminf_{z\to 0}  
\psi(z) - \gamma \log |z|^2.
$$
\end{theorem}

\begin{proof} Since $\psi$ is $\Gamma$ invariant, we know that $\psi$ has isolated order $\gamma$ log poles at $\Gamma$, one may use Ohsawa--Takegoshi extension theorem to extend $L^2$ functions from $\Gamma$ to $\mathbb C^n$. Denote by $F$ the extension with minimal $L^2$ norm. By our assumption
$$
i\partial\dbar \left(\frac{n}{\gamma} \psi(z)+\pi|z|^2\right)\geq
 \left(1-\frac n\gamma\right) i\partial\dbar (\pi|z|^2).
$$
Hence \cite[Theorem 3.8]{BL} or the main theorem in \cite{GZ} implies
$$
\left(1-\frac n\gamma\right) \int_{\mathbb C^n}|F|^2 e^{-\pi |z|^2} \leq  \frac{\pi^n}{n!} \limsup_{z\to 0} e^{-\frac n\gamma(\psi(z)-\gamma \log|z|^2)},
$$
thus our theorem follows.
\end{proof}

\subsection{Transcendental lattices and jet interpolations}\label{se:next} Let us first introduce the following definition for jet interpolations. 

\begin{definition}\label{de:k-int} Let $k\geq 0$ be an integer. Let $\Gamma$ be a lattice in $\mathbb C^n$. Put
$$
N_k:=\{\alpha=(\alpha_1, \cdots, \alpha_n)\in \mathbb Z_n: \alpha_j\geq 0, \ \sum \alpha_j \leq k\}.
$$
We say that $\Gamma$ is a set of $k$-jet interpolation for $\mathcal F^2$ if there exists a constant $C>0$ such that for every sequence of complex numbers $\{c_{\lambda, \alpha}\}_{\lambda\in \Gamma, \alpha\in N_k }$ with 
$
\sum_{\lambda\in \Gamma, \alpha\in N_k} |c_{\lambda, \alpha}|^2 e^{-\pi |\lambda|^2}=1,
$
there exists $F\in \mathcal F^2$ with
$$
\left(e^{\pi |z|^2}\partial^{\alpha} (e^{-\pi |z|^2} F)\right)  \big|_{z=\lambda} = c_{\lambda,\alpha}, \ \  \forall \,\lambda\in\Gamma, \,\alpha\in N_k, \ \ \partial^\alpha f :=\frac{\partial^{\alpha_1+\cdots+\alpha_n} f}{\partial z_1^{\alpha_1}\cdots \partial z_n^{\alpha_n}},
$$ 
and $||F||^2 \leq C$.
\end{definition}

The proof of the H\"ormander criterion above also implies the following result. 

\begin{theorem}\label{th:jet} Let $k\geq 0$ be an integer. Let $\Gamma$ be a transcendental lattice in $\mathbb C^n$. Assume that
\begin{equation}\label{eq:jet}
|\Gamma| > \frac{(n+k)^n}{n!},
\end{equation}
then $\Gamma$ is a set of $k$-jet interpolation for $\mathcal F^2$.
\end{theorem}

\begin{proof} Since $\Gamma$ is transcendental, by \eqref{eq:Dem}, we know that \eqref{eq:jet} implies that  there exists a non positive $\Gamma$ invariant function $\psi$ on $\mathbb C^n$ such that $\psi(z)+\pi|z|^2$ is plurisubharmonic on $\mathbb C^n$, $\psi$ is smooth outside $\Gamma$ and $
\psi(z) - \gamma \log |z|^2$
is bounded near the origin for some constant $\gamma>n+k$. Thus the H\"ormander $L^2$ estimate with singular weight $\psi$ (similar to the proof of the H\"ormander criterion above) gives the above theorem.
\end{proof}

\noindent
\textbf{Remark}: \emph{If $\Gamma$ is transcendental with $|\Gamma|>\frac{n^n}{n!}$ then
$$
\big| \sqrt{\frac{n+k}n} \,\Gamma\big|> \frac{(n+k)^n}{n!},
$$
thus we know that $\sqrt{\frac{n+k}n} \,\Gamma$ is  a set of $k$-jet interpolation for $\mathcal F^2$. In one dimensional case, we have the following theorem \cite{GL0}.}

\begin{theorem}\label{th:jet1}  Let $\Gamma$ be a  lattice in $\mathbb C$. Then the followings are equivalent:
\begin{itemize}
\item[(1)] $\Gamma$ is a set of interpolation for $\mathcal F^2$;
\item[(2)] $\sqrt{k+1}\,\Gamma$ is a set of $k$-jet interpolation for $\mathcal F^2$ for some positive integer $k$;
\item[(3)] $|\Gamma|>1$. 
\end{itemize}
\end{theorem}

For the higher-dimensional cases, we can prove the following result.

\begin{theorem}\label{th:pre} Let $\Gamma$ be a transcendental lattice in $\mathbb C^n$.  If  $\sqrt{\frac{n+k}n} \,\Gamma$ is  a set of $k$-jet interpolation for $\mathcal F^2$ for some non-negative integer $k$ then $|\Gamma| \geq \frac{(k+1)^n}{(n+k)^n}\frac{n^n}{n!}$.
\end{theorem}

\begin{proof} Put $\nabla^\alpha F:=e^{\pi |z|^2}\partial^{\alpha} (e^{-\pi |z|^2} F)$ and $\Gamma_k:=\sqrt{\frac{n+k}n} \,\Gamma$. Assume that $\sqrt{\frac{n+k}n} \,\Gamma$ is  a set of $k$-jet interpolation for $\mathcal F^2$. Let us define
\begin{equation}\label{eq:pre}
G(z):=\sup\{|F(z)|^2e^{-\pi|z|^2}: F\in \mathcal F^2, \ \nabla^\alpha F(\lambda)=0, \ \forall \ \lambda\in \Gamma_k\}.
\end{equation}
Then by the Balian-Low type theorem, we know that $G$ is not identically zero on $\mathbb C^n$. We claim that $G$ is $\Gamma_k$ invariant. In fact, if we put
$$
T_\lambda F(z):=F(z+\lambda) e^{-\pi|\lambda|^2/2}e^{-\pi z\bar\lambda}.
$$
Then $F \mapsto T_\lambda F$ is an isomorphism on $\mathcal F^2$ with $||F||=||T_\lambda F||$, $|T_\lambda F(0)|^2=|F(\lambda)|^2e^{-\pi|\lambda|^2}$ and
$$
\nabla^\alpha (T_\lambda F)(0)=0 \Leftrightarrow \nabla^\alpha F(\lambda)=0.
$$
Hence $G$ is $\Gamma_k$ invariant. Put $\psi=\log G$, we know that $\psi$ is $\Gamma$ invariant, $\psi(z)+\pi|z|^2$ is plurisubharmonic on $\mathbb C^n$ and $\psi(z)-(k+1)\log |z|^2$ is bounded above near $z=0$. Since $\Gamma$ is transcendental, we know that $\Gamma_k$ is also transcendental, thus the $\Gamma_k$ invariant analytic set $\{\psi=-\infty\}$ is discrete. Hence \eqref{eq:Dem} gives that
$$
(n! |\Gamma_k|)^{1/n} \geq k+1,
$$
from which our theorem follows.
\end{proof}

\noindent
\textbf{Remark}: \emph{The above theorem is our motivation for the conjecture A after Theorem \ref{th:main-01}, moreover, notice that $\lim_{k\to \infty} \frac{k+1}{n+k}=1$, the above theorem also suggests the following higher-dimensional analogue of Theorem \ref{th:jet1}.}

\medskip
\noindent
\textbf{Conjecture B}: \emph{Let $\Gamma$ be a transcendental lattice in $\mathbb C^n$.  Then the followings are equivalent:
\begin{itemize}
\item[(1)] $\Gamma$ is a set of interpolation for $\mathcal F^2$;
\item[(2)] $\sqrt{\frac{n+k}n} \,\Gamma$ is  a set of $k$-jet interpolation for $\mathcal F^2$ for some positive integer $k$;
\item[(3)] $|\Gamma|>\frac{n^n}{n!}$. 
\end{itemize}}

\medskip

\noindent
\textbf{Remark}: \emph{From Proposition \ref{pr:1.1}, we know that the above conjecture implies conjecture $A$.}

\section{H\"ormander constants and K\"ahler embeddings} 

\subsection{H\"ormander constants and proof of Proposition \ref{pr:hor}}\label{ss:dem}  In Definition \ref{de:hor} and Definition \ref{de:beta} we have defined the H\"ormander constants and the $\beta$-Seshadri constants for an $n$-dimensional compact K\"ahler manifold $(X, \omega)$. If all $\beta_j=1/n$ and $\omega \in c_1(L)$ for some ample line bundle $L$ then we have
\begin{equation}\label{eq:D-beta}
n \epsilon_x(\omega;\beta)= 
\epsilon_x(\omega),
\end{equation}
where $\epsilon_x(\omega)$ denotes  the Seshadri constant introduced by Demailly in \cite{Dem-sin} (in fact, from (6.2) in \cite{Dem-sin}, we have $n \epsilon_x(\omega;\beta)=\gamma(L,x)$, but Theorem 6.4 in \cite{Dem-sin} tells us that $\gamma(L,x)$ is precisely the Seshadri constant used in algebraic geometry when $L$ is ample). In general, the condition $\sum \beta_j=1$ is used to make sure that
$$
\sup\{c\geq 0: e^{-cT_\beta} \ \text{is integrable near $z=0$}\}=1.
$$

\medskip

\noindent
\textbf{Remark.}  \emph{For transcendental $\omega$ on a general compact K\"ahler manifold, we know that (see Theorem \ref{th:CEL} below for the proof and generalizations) $ n \epsilon_x(\omega;\beta) 
$ is equal to the generalized Seshadri constant (also denoted by $\epsilon_x(\omega)$) defined by Tosatti in \cite[section 4.4]{Tosatti}. In this general case, we shall prove the following result.}

\begin{proposition}\label{pr:hor2} Let $(X, \omega)$ be an $n$-dimensional compact K\"ahler manifold. Assume that $X$ has no non-trivial analytic subvarieties, then
\begin{align}\label{eq:Dem}
\sup_{\beta\in \mathcal B} \epsilon_x(\omega; \beta) = \frac{\epsilon_x(\omega)}n =\frac{ \left(\int_X \omega^n\right)^{1/n}}{n}.
\end{align}
\end{proposition}

\begin{proof} Note that for every $\beta\in \mathcal B$, by \cite[page 167, Corollary 7.4]{Demailly12} (our definition of $dd^c$ in Definition \ref{de:beta} is half of the one there) we have
\begin{equation}\label{eq:Dem1}
(dd^c T_\beta)^n_x =(\beta_1\cdots \beta_n)^{-1}, \ \ \ \ (dd^c T_\beta)^n_x:=\lim_{r\to 0} \int_{|z-x|<r} (dd^c T_\beta)^n,
\end{equation}
which gives
$$
\int_X \omega^n \geq \epsilon_x(\omega; \beta)^n  (\beta_1\cdots \beta_n)^{-1} \geq \epsilon_x(\omega; \beta)^n  n^n.
$$
Hence
\begin{equation}\label{eq:Dem2}
\sup_{\beta\in \mathcal B} \epsilon_x(\omega; \beta) \leq \frac{ \left(\int_X \omega^n\right)^{1/n}}{n}.
\end{equation}
On the other hand, we have the following identity proved by Tosatti in \cite[Theorem 4.6]{Tosatti}
$$
\epsilon_x(\omega)=\inf_{V\ni x} \left( \frac{\int_V \omega^{\dim V}}{{\rm mult}_x V}\right)^{\frac1{\dim V}},
$$
where the infimum runs over all positive-dimensional irreducible analytic subvarieties $V$ containing $x$ and ${\rm mult}_x V$ denotes the multiplicity of $V$ at $x$. Hence if $X$ has no non-trivial subvarieties then (put $\beta_0= (1/n ,\cdots, 1/n)$, use \eqref{eq:D-beta} and the remark above)
$$
\sup_{\beta\in \mathcal B} \epsilon_x(\omega; \beta)  \geq  \epsilon_x(\omega; \beta_0) =\frac{\epsilon_x(\omega)}n = \frac{ \left(\int_X \omega^n\right)^{1/n}}{n}.
$$
The above inequality and  \eqref{eq:Dem2} together imply \eqref{eq:Dem}. 
\end{proof}

\begin{proof}[Proof of Proposition \ref{pr:hor}] Apply \eqref{eq:Dem} to the case that $X=\mathbb C^n /\Gamma$ and $\omega = dd^c(\pi |z|^2)$, we get immediately Proposition \ref{pr:hor} (note that in this case
$$
\sup_{\beta\in \mathcal B} \epsilon_x(\omega; \beta) = \iota_\Gamma
$$
and $\int_X \omega^n = n! \,|\Gamma|$).
\end{proof}

\subsection{Relation with the $s$-invariant} Our $\beta$-Seshadri constant is closely related to the $s$-invariant introduced by Cutkosky, Ein and Lazarsfeld in \cite{CEL}. 

\begin{theorem}\label{th:CEL} Let $(X, \omega)$ be an $n$-dimensional compact K\"ahler manifold. Assume that $\omega$ lies in the first Chern class of a holomorphic line bundle $L$ on $X$. Fix $\beta=(\beta_1,\cdots, \beta_n)\in \mathbb R^n$ such that all $\beta_j^{-1}$ are positive integers. Then 
\begin{equation}\label{eq:CEL}
\epsilon_x(\omega;\beta)=\frac1{s_{L}(\mathcal I_\beta)} ,
\end{equation}
where $\mathcal I_\beta$ is the ideal of $\mathcal O_X$ generated by $\{z_1^{1/\beta_1}, \cdots, z_n^{1/\beta_n}\}$ and 
$$
s_{L}(\mathcal I_\beta):=\min\{s\in \mathbb R: \mu^*(sL)-E \ \text{is nef}\}
$$
is the $s$-invariant of $\mathcal I_\beta$ with respect to $L$ (see Definition 5.4.1 in \cite{Laz}), where $\mu$ is the blowing-up of $X$ along $\mathcal I_\beta$ with exceptional divisor $E$.
\end{theorem}

\begin{proof} First let us prove $\epsilon_x(\omega;\beta)\geq \frac1{s_{L}(\mathcal I_\beta)}$. Note that $L$ is ample since $\omega$ is positive, hence
$$
\frac1{s_{L}(\mathcal I_\beta)}= \sup\{\gamma\geq 0: \mu^*L- \gamma E \ \text{is ample}\}.
$$
By Example 5.4.10 in \cite{Laz}, one may replace $\mu$ by a desingularization $f: Y\to X$ of $\mathcal I_\beta$ with exceptional divisor $F$, more precisely, we have
$$
\frac1{s_{L}(\mathcal I_\beta)}= \sup\{\gamma\geq 0: f^*L- \gamma F \ \text{is ample}\}.
$$
which implies that for every $\gamma <\frac1{s_{L}(\mathcal I_\beta)}$ there is a singular metric $e^{-\phi}$ on $f^*L$ with $\gamma$-log pole along $F$ such that $i\partial\dbar \phi>0$ on $Y$. Then the weight $f_*\phi$ on $L$ will have the $\gamma T_\beta$-singularity, from which we know that $\epsilon_x(\omega;\beta)\geq \gamma$. Hence $\epsilon_x(\omega;\beta)\geq \frac1{s_{L}(\mathcal I_\beta)}$. 

Now let us prove that $\epsilon_x(\omega;\beta)\leq \frac1{s_{L}(\mathcal I_\beta)}$. For every $\gamma < \epsilon_x(\omega;\beta)$, we can find a singular metric $e^{-\psi}$ on $L$ with $\gamma T_\beta$-singularity such that $i\partial\dbar \psi>0$. Then $e^{-f^*\psi}$ defines a singular metric on $f^*L$ with $\gamma$-log pole along $F$ such that $i\partial\dbar (f^*\psi)>0$ on $Y$, from which we know that $\epsilon_x(\omega;\beta)\leq \frac1{s_{L}(\mathcal I_\beta)}$.
\end{proof}

\noindent
\textbf{Remark.}  \emph{Since $\mathcal I_\beta$ are special monomial ideals, it is not hard to find the explicit desingularizations. Let us look at the simple example $\beta=(1,\frac12)$. Then, in this case, we have
$$
\mathcal I_\beta= {\rm Span}\{z_1, z_2^2\}.
$$
First, one may blow up the origin, so $z_1=uv, z_2=u$ gives
$$
{\rm Bl}_0 \mathcal I_\beta = {\rm Span}\{u v, u^2\}, 
$$
then we can blow up the point $(u, v)=(0,0):=\tilde 0$, so $v= ts, u=s$ gives
$$
{\rm Bl}_{\tilde 0}{\rm Bl}_0 \mathcal I_\beta = {\rm Span}\{s^2\},
$$
from which we know that $F= 2\cdot |\{s=0\}|$.}

\subsection{Extremal property of the $\beta$-Seshadri constant} For $s$-invariant of a general monomial ideal 
$$
\mathcal I_P:= {\rm Span}\{z^{\alpha^1}, \cdots, z^{\alpha^k}\}, \ \ \alpha^j \in \mathbb Z_{>0}^n, \ \ z^{\alpha^j}:=z_1^{\alpha^j_1} \cdots z_n^{\alpha^j_n},
$$ 
with isolated zero set $\{x\}$, where $P$ is the Newton polytope defined by 
$$
P:= \text{convex hull of $\cup_{1\leq j \leq k} P_{\alpha^j}$}, \ \ P_{\alpha^j}:=\{x\in \mathbb R^n: x_l \geq \alpha^j_l, \ 1\leq l\leq n \}, 
$$
one may correspondingly define the $P$-Seshadri constant
\begin{align}\label{eq:P-se}
\epsilon_x(\omega;P):= \sup\{\gamma\geq 0: &  \ \  \text{there exists an  $\omega$-psh function $\psi$} \\
\nonumber & \text{on $X$ with $\psi=\gamma T_P$ near $x$}\},
\end{align}
where
$$
T_P:= \log \sum_{\alpha \in v(P)}|z^\alpha|^2, \ \ v(P) \ \text{denotes the set of vertices of $P$}.
$$
Then the proof of Theorem \ref{th:CEL} also implies
\begin{equation}\label{eq:CEL2}
\epsilon_x(\omega;P)=\frac1{s_{L}(\mathcal I_P)}.
\end{equation}
The reason why we only use $\beta$-Seshadri constants in this paper is that they have the following extremal property. 

\begin{theorem}\label{th:extremal} $\sup\{ \epsilon_x(\omega;P): \text{$(1,\cdots , 1)$ lies in the boundary of $P$}\}$ equals $\sup_{\sum \beta_j=1}\epsilon_x(\omega;\beta)$.
\end{theorem}

\begin{proof} The proof follows from a very simple fact: for an arbitrary Newton polytope $P$  such that $(1,\cdots, 1)$ lies in the boundary of $P$, one can always find $\beta\in \mathbb R_{>0}^n$ with $\sum \beta_j =1$ and
$$
P\subset P_\beta :=\{x\in \mathbb R_{\geq 0}^n: \beta_1 x_1+\cdots +\beta_n x_n \geq 1\},
$$
hence
$$
\epsilon_x(\omega;P) \leq \epsilon_x(\omega; P_\beta) =  \epsilon_x(\omega; \beta)
$$
gives the theorem.
\end{proof}

\noindent
\textbf{Remark.}  \emph{The condition that $(1,\cdots, 1)$ lies in the boundary of $P$ is equivalent to the following identity
$$
\sup\{c\geq 0: e^{-cT_P} \ \text{is integrable near $z=0$}\}=1,
$$
see \cite{Ho, Gu} for the proof and related results.}

\subsection{McDuff--Polterovich's theorem}  McDuff--Polterovich \cite{MP} proved that the Seshadri constant
\begin{align*}
\epsilon_x(\omega):= & \sup\{\gamma\geq 0:  \text{there exists an  $\omega$-psh function $\psi$}  \\ 
& \text{on $X$ with $\psi=\gamma \log|z|^2$ near $x$}\}
\end{align*}
 is always no bigger than the following \emph{Gromov width} of $(X, \omega)$ (see \cite[Theorem 1.1]{LMS} and \cite{EV})
$$
c_G(X,\omega):=\sup \{\pi r^2: B_r \xhookrightarrow{Symplectic} (X, \omega)  \}, \ \ \ B_r:=\{z\in \mathbb C^n: |z|<r\} ,
$$
where $B_r \xhookrightarrow{Symplectic} (X, \omega)$ means there exist a \emph{smooth} injection $
f: B_r \hookrightarrow   X
$
such that $
f^*(\omega)=\omega_{\rm euc}:=\frac i2 \sum_{j=1}^n dz_j \wedge d\bar z_j.$
In fact, the proof in \cite{MP} also gives the following stronger result:

\begin{theorem}\label{th:MP} Let $(X, \omega)$ be a compact K\"ahler manifold. Denote by $\mathcal K_\omega$ the space of K\"ahler metrics in the cohomology class $[\omega]$. 
Then the Seshadri constant $\epsilon_x(\omega)$ is equal to the following K\"ahler width
\begin{align*}\label{eq:TSes}
c_x(\omega):=\sup \left\lbrace\pi r^2: B_r \xhookrightarrow{hol_x} (X, \tilde{\omega}), \ \exists\ \text{$\tilde{\omega} \in \mathcal K_\omega$}  \right\rbrace,
\end{align*}
where
$$
B_r:=\left\lbrace z\in \mathbb C^n: \sum_{j=1}^n |z_j|^2 <r^2 \right\rbrace,
$$  
"$B_r \xhookrightarrow{hol_x} (X, \tilde{\omega})$" means that there exists an holomorphic injection $f: B_r \hookrightarrow  X$
such that $
f(0)=x$ and  
$f^*(\tilde\omega)=\omega_{\rm euc}$.
\end{theorem}

\begin{proof}  The proof here is different from McDuff--Polterovich's approach in \cite{MP}. Our main idea is to use the following plurisubharmonic function on $\mathbb C^n$
\begin{equation}\label{eq:psi-r}
\psi_r(z):= 
\begin{cases}
\pi r^2 \left( \log \frac{|z|^2}{r^2} +1\right) & |z|<r \\ 
\pi |z|^2 & |z| \geq r,
\end{cases}
\end{equation}
which satisfies
\begin{itemize}
\item[(L1)] $\psi_r(z)-\pi r^2 \log|z|^2$ is bounded near $0$;
\item[(L2)] $\psi_r(z) \leq \pi |z|^2$ on $\mathbb C^n$ and $\{z\in \mathbb C^n:\psi_r(z)<\pi|z|^2\}=B_r$;
\item[(L3)] For every $w\in B_r$ and $0<\delta<1$ we have
$$
\psi_r(f_\delta(w))= \psi_r(w) + \pi r^2 \log(\delta^2),
$$
where $
f_\delta(w):=(\delta w_1, \cdots, \delta w_n).
$
\end{itemize}
We shall use (L1)-(L3) to prove (P1), (P2) below which imply the theorem: 
\begin{itemize}
\item[(P1)] If $c_x(\omega) > \pi r^2$ then $\epsilon_x(\omega) \geq \pi r^2$;
\item[(P2)] If $\epsilon_x(\omega) > \pi r^2$ then $c_x(\omega) \geq \pi r^2$.
\end{itemize}

\medskip
\noindent
\textbf{Proof of (P1)}:  If $c_x(\omega) > \pi r^2 $ then then one may think of $B_r$ as a K\"ahler subset of $X$. Let us define
$\tilde \psi$ such that $\tilde \psi=\psi_r-\pi |z|^2$ on $B_r$ and $\tilde\psi=0$ outside $B_r$ in $X$. Then (L2) implies that $\tilde \psi$ is $\omega$-psh (note that $\omega_{\rm euc}=dd^c(\pi|z|^2)$). Fix a small $\varepsilon>0$ (assume that $\varepsilon<\pi r^2$), then by (L1) one may take a sufficiently big $C>0$ such that
\begin{equation}\label{eq:glue}
\psi(z):=\max\{\tilde\psi(z)+C, (\pi r^2-\varepsilon)\log |z|^2\}
\end{equation}
is equal to $\tilde\psi+C$ near the boundary of $B_r$. Hence $\psi$ extends to an $\omega$-psh function on $X$. Notice that $\psi= (\pi r^2-\varepsilon)\log |z|^2 $ near $z(x)=0$, hence $\epsilon_x(\omega) \geq \pi r^2-\varepsilon$  for all $\varepsilon>0$ and (P1) holds.

\medskip
\noindent
\textbf{Proof of (P2)}: To simplify the notation we take $r=1$. If  $\epsilon_x(\omega) > \pi$ then one may use  the construction in \eqref{eq:glue} to produce an $\omega$-psh function $\psi$ smooth on $X\setminus\{x\}$ such that 
$$
(*) \qquad \text{$\omega+dd^c\psi=dd^c (\psi_1+|z|^2)$ near $z(x)=0$}, 
$$
and $\omega+dd^c\psi >0$ on $X$, where $\psi_1$ in defined in \eqref{eq:psi-r}. For $w\in B_1$ and $0<\delta<1$ let us define
$$
\phi(f_\delta(w)):= {\rm Max} \left\lbrace  \psi_1(f_\delta(w)) +|f_\delta(w)|^2,  \pi (|w|^2+2\log \delta)\right\rbrace,
$$
where ${\rm Max} $ denotes a regularized max function. By (L3) and (L2)
$$
 \psi_1(f_\delta(w) ) = \psi_1(w ) +\pi\log (\delta^2) \leq\pi( |w|^2+2\log \delta)
$$
for $w\in \overline{B_1}$, with identity holds if and only if $w\in \partial B_1$; together with $(*)$, we know that for every $0<\gamma<1$, one may take a small $\delta$ such that $dd^c \phi$ extends to a K\"ahler form $\tilde\omega \in \mathcal K_\omega$ with 
$$
f_\delta^*(\tilde\omega)=f_\delta^* dd^c \phi = dd^c(\phi (f_\delta(w)))=dd^c(\pi |w|^2) \  \ \text{on} \ B_{\gamma}.
$$ 
Letting $\gamma \to 1$ we finally get $c_x(\omega)\geq \pi$. The proof is complete.
\end{proof}

In order to generalize the above proof to general $\beta$-Seshadri constants, we need to construct the associated $\beta$-version of $\psi_r$ in \eqref{eq:psi-r} (see Lemma \ref{le:envelope} below). In the next subsection, we shall use the Legendre transform theory to decode the construction.

\subsection{Iterated Legendre transform and proof of Theorem A} The main ingredient in our proof of Theorem A is the theory of iterated Legendre transform.
\begin{definition} Let $\phi$ be a  smooth convex function on $\mathbb R^n$. We call
$$
\phi^*(\alpha):=\sup_{y\in \mathbb R^n} \alpha\cdot y-\phi(y)
$$
the Legendre transform of $\phi$. Let $A\subset \mathbb R^n$ be a closed set. We call
$$
\phi_A(x):=\sup_{\alpha \in A} \alpha\cdot  x- \phi^*(\alpha) 
$$
the iterated Legendre transform of $\phi$ with respect to $A$.
\end{definition}

\noindent
\textbf{Remark.} \emph{Notice that 
$$
\phi(x)+ \phi^*(\alpha) \geq \alpha\cdot x,
$$
hence we always have
$$
\phi_A \leq \phi.
$$}

\begin{lemma} Let $\phi$ be a smooth convex function on $\mathbb R^n$. Then
$$
\phi(x)+ \phi^*(\alpha)= \alpha\cdot x
$$
if and only if $\alpha=\nabla \phi (x)$, where
$$
\nabla \phi (x):=\left(\frac{\partial \phi}{\partial x_1}(x), \cdots, \frac{\partial \phi}{\partial x_n}(x) \right).
$$
\end{lemma}

\begin{proof} It follows from the fact that $x$ is the maximum point of the following concave function
$$
\psi^\alpha: y \mapsto \alpha\cdot y-\phi(y)
$$
if and only if $x$ is a critical point of $\psi^\alpha$.
\end{proof}

\begin{proposition}\label{pr:legendre} If $\phi$ is smooth strictly convex and  $A \subset \mathbb R^n$ is closed then
$$
\{x\in \mathbb R^n: \phi(x)=\phi_A (x)\} = \{x\in \mathbb R^n: \nabla \phi(x) \in A\}.
$$
\end{proposition}

\begin{proof} By the above lemma, we have
$$
\alpha\cdot x -\phi^*(\alpha) \leq \phi(x)
$$
with identity holds if and only if $\alpha=\nabla \phi(x)$. Hence  $\nabla \phi(x)$ is the unique maximum point of the following function
$$
\rho^x: \alpha\mapsto \alpha\cdot x -\phi^*(\alpha).
$$
The proof of Proposition 2.2 in \cite{Wang-AF} implies that $\rho^x$ is smooth strictly concave on $\nabla \phi(\mathbb R^n)$. Hence the supremum of $\rho^x$ on the complement of any small ball around $\nabla \phi(x)$ must be strictly smaller than $\phi(x)$. By our assumption, $A$ is closed, thus $\phi(x)=\phi_A (x)$ if and only if $\nabla \phi(x) \in A$.
\end{proof}

\begin{definition} Let $\phi$ be a smooth strictly convex function on $\mathbb R^n$ and $A$ be a closed set in $\mathbb R^n$. We call
$$
\Omega_A(\phi):=\{x\in \mathbb R^n: \phi_A (x)<\phi(x)\} 
$$
the Hele--Shaw domain of $(\phi, A)$.
\end{definition}

The above proposition implies that
\begin{equation}\label{eq:hsd}
\Omega_A(\phi)= (\nabla \phi)^{-1} (\mathbb R^n\setminus A).
\end{equation}
Our key observation is the following:

\begin{theorem}\label{le:key} Assume that  $\phi$ is smooth strictly convex and  $A \subset \mathbb R^n$ is closed.  If $x\in \Omega_A(\phi)
$ then 
\begin{equation}\label{eq:key}
\phi_A(x)=\sup_{\alpha \in \partial A} \alpha\cdot x-\phi^*(\alpha),
\end{equation}
where $\partial A$ denotes the boundary of $A$.
\end{theorem}

\begin{proof} Since $\nabla \phi(x)$ is the unique maximum point of the following concave function
$$
\rho^x: \alpha\mapsto \alpha\cdot x -\phi^*(\alpha)
$$
and $\nabla \phi(x) \notin A$, we know that for every point $a\in A$, 
$$
f(t):= \rho^x (t\nabla \phi(x) +(1-t)a )
$$
is increasing on $t\in [0,1]$. Take $\hat t \in [0,1]$ such that 
$$
\hat t\nabla \phi(x) +(1-\hat t)a \in \partial A,
$$ 
we know that
$$
\rho^x(a) =f(0)  \leq f(\hat t) \leq  \sup_{\alpha \in \partial A} \rho^x(\alpha).
$$
Hence the theorem follows.
\end{proof}

\begin{lemma}\label{le:envelope} With the notation in Theorem A, for each $r>0$ there exists $\psi_r \in {\rm psh}(\mathbb C^n)$ such that
\begin{itemize}
\item[(L1)] $\psi_r-\pi r^2 T_\beta$ is bounded near $0$;
\item[(L2)] $\psi_r(z) \leq \pi |z|^2$ on $\mathbb C^n$ and $\{z\in \mathbb C^n: \psi_r(z)<\pi|z|^2\}=B_r^\beta$;
\item[(L3)] For every $w\in B_r^\beta$ and $0<\delta<1$ we have
$$
\psi_r(f_\delta(w))= \psi_r(w) + \pi r^2 \log(\delta^2),
$$
where $
f_\delta(w):=(\delta^{\beta_1}w_1, \cdots, \delta^{\beta_n}w_n).
$
\end{itemize}
\end{lemma}

\begin{proof} Let us consider
$$
\psi_r(z):=\phi_A(\log|z_1|^2, \cdots, \log|z_n|^2), 
$$
where
$$
\phi(x):=\pi (e^{x_1}+\cdots+e^{x_n}), \ \ A:=\{\alpha \in \mathbb R^n: \alpha\cdot \beta \geq \pi r^2\}.
$$
Then we have
$$
\nabla \phi(x) =\pi(e^{x_1}, \cdots, e^{x_n})
$$
and
$$
\phi^*(\alpha)=
\begin{cases}
\sum_{j=1}^n \left(\alpha_j \log \frac{\alpha_j}{\pi}-\alpha_j \right)  &  \ \forall \ \alpha_j \geq 0 \\
\infty &  \ \exists\ \alpha_j < 0,
\end{cases}
$$
where  $0\log 0:=0$. Hence
$$
\phi_A(x)=\sup_{\alpha\in A_+} \alpha\cdot x-\sum_{j=1}^n \left(\alpha_j \log \frac{\alpha_j}{\pi}-\alpha_j \right),
$$ 
where
$$
A_+:=\{\alpha \in [0, \infty)^n: \alpha\cdot \beta \geq \pi r^2\}.
$$

\medskip
\noindent
\textbf{Proof of (L1)}: Notice that $\psi_r-\pi r^2 T_\beta$ is  bounded near $0$ if and only if $\phi_A(x)-\sup_{\alpha \in A_+} \alpha\cdot x$ is bounded on $(-\infty, 0]^n$. The above formula for $\phi_A$ implies that
$$
\inf_{1\leq j\leq n} \left(\frac{\pi r^2}{\beta_j}-\frac{\pi r^2}{\beta_j} \log \frac{ r^2}{\beta_j}\right) \leq \phi_A(x)-\sup_{\alpha \in A_+} \alpha\cdot x \leq  n\pi
$$
for all $x\in (-\infty, 0]^n$. Hence (L1) follows.

\medskip
\noindent
\textbf{Proof of (L2)}: Follows directly from \eqref{eq:hsd}.

\medskip
\noindent
\textbf{Proof of (L3)}: Notice that \eqref{eq:key} implies
$$
\phi_A(x+(\log (\delta^2))\beta)=\phi_A (x)+ \pi r^2 \log(\delta^2), \ \ \ \forall \ x\in \Omega_A(\phi),
$$
from which (L3) follows.
\end{proof}

\begin{proof}[Proof of Theorem A] Similar as the proof of Theorem \ref{th:MP} (replace $B_r$ by $B_r^\beta$), we know that the lemma above gives Theorem A. 
\end{proof}

\subsection{A partial converse of the H\"ormander criterion} 

For general higher dimensional cases,  we do not know whether the H\"ormander criterion is an equivalent criterion or not. Based on Demailly--P\u{a}un's generalized  Nakai--Moishezon ampleness criterion \cite{DP}, a theorem of Nakamaye \cite{Nak, Oht}, Lindholm's result \cite{Lindholm} and the  
\emph{Balian-Low type theorem} (see \cite[Theorem 10]{H07} and \cite[Theorem 1.5]{AFK}), we obtain the following partial converse of the H\"ormander criterion.

\begin{theorem}\label{th:converse}
Let $\Gamma$ be a lattice in $\mathbb C^n$. Denote by $\iota_\Gamma$ the H\"ormander constant  of  $(\mathbb C^n/\Gamma, \omega_{\rm euc})$.
\begin{itemize}
\item[(1)] If  $\Gamma$ is a set of interpolation for $\mathcal F^2$ and all irreducible analytic subvarieties of $X$ are translates of complex tori then $\iota_\Gamma>1/n$;
\item[(2)] If  $\Gamma$ is a set of interpolation for $\mathcal F^2$ and the only positive dimensional irreducible analytic subvariety of $X$ is $X$ itself then 
$$
\iota_\Gamma=\frac{(n!\,|\Gamma|)^{1/n}
}{n} > \frac{(n!)^{1/n}
}{n};
$$
\item[(3)] Assume that $\omega_{\rm euc}$ is rational on $\Gamma$ or the Picard number of $X$ is $n^2$. If  $\Gamma$ is a set of interpolation for $\mathcal F^2$ then  $\iota_\Gamma>1/(n\,e)$.
\end{itemize}
\end{theorem}

\begin{proof} Denote by $0$ the unit of the torus and write $\omega:=\omega_{\rm euc}$.
The main idea is to use the following Demailly--P\u{a}un identity proved by Tosatti in \cite[Theorem 4.6]{Tosatti}
$$
\epsilon_0(\omega)=\inf_{V\ni 0} \left( \frac{\int_V \omega^{\dim V}}{{\rm mult}_0 V}\right)^{\frac1{\dim V}},
$$
where the infimum runs over all positive-dimensional irreducible analytic subvarieties $V$ containing $0$, and ${\rm mult}_0 V$ denotes the multiplicity of $V$ at $0$. Now let us prove Theorem (1), by our assumption, it suffices to show that for all complex subtorus
$$
\int_V \omega^{d}  >1,   \ \ d:=\dim V.
$$
Choose a $\mathbb C$ linear subspace $E$ of $\mathbb C^n$ such that $V=E/(E\cap \Gamma)$.
Notice that if $\Gamma$ is a set of interpolation for $\mathcal F^2$, then $E\cap \Gamma$ is a set of interpolation for $\mathcal F^2|_{E}$. Thus the Balian--Low type theorem implies  that
$$
|V|=\int_V \omega^{d} /d! >1,
$$
which gives $\epsilon_0(\omega)>1$, hence (1) follows.  Now assume further that $X$ has no non-trivial subvarieties, then \eqref{eq:Dem} implies
$$
(n\iota_\Gamma)^n/n!=  |X|=|\Gamma|, 
$$
thus (2) follows directly from the Balian--Low type theorem. To prove  (3), we shall use the following inequality (see \cite[Lemma 3.2]{Ito})
\begin{equation}\label{eq:Nak}
\epsilon_0(\omega)\geq \inf_{V\ni 0} \frac{\left( \int_V \omega^{\dim V}\right)^{\frac1{\dim V}}}{\dim V}, \ \ \text{if $\omega$ is integral on $\Gamma$},
\end{equation}
where the infimum runs over all positive-dimensional abelian subvarieties $V$ containing $0$. Notice that the right hand side of \eqref{eq:Nak} is 1-homogeneous with respect to $\omega$, we know that \eqref{eq:Nak} also holds for all $\omega=c\omega'$, where $\omega'$ is integral on $\Gamma$. In particular, it holds true if $\omega$ is rational on $\Gamma$. In case the Picard number of $X$ is $n^2$, we know that $\omega$ can be approximated by rational $\omega'$, hence \eqref{eq:Nak} is true for all $\omega$. By Balian--Low type theorem, we have
$$
\int_V \omega^{d}  > d! .
$$
By Stirling's approximation, we have 
$$
(d!)^{1/d}/d \geq e^{-1},
$$
hence (3) follows. 
\end{proof}

\subsection{H\"ormander constants and densities of  general discrete sets}

\begin{definition}\label{de:shsha-2} Let $S$ be a discrete set in $\mathbb C^n$. Let $\psi$ be a non positive function such that $\psi+\pi|z|^2$ is plurisubharmonic on $\mathbb C^n$. Let $\gamma$ be positive number. We call $(\psi, \gamma)$ an $S$-admissible pair if $\psi$ is smooth outside $S$, $e^{-\psi/\gamma}$ is not integrable near every point in $S$ and there exists a small constant $\varepsilon_0>0$ such that
\begin{equation}\label{eq: eM}
\inf_{\varepsilon<|z-\lambda|<2\varepsilon \ \text{for some} \ \lambda \in S}  \psi(z) >-\infty, \qquad \ \forall \ 0<\epsilon\leq\epsilon_0.
\end{equation}
Assume that there exists an $S$-admissible pair, then we call
$$
\iota(S):=\sup\{\gamma >0: \text{there exists $\psi$ such that $(\psi, \gamma)$ is $S$-admissible}\}
$$
the H\"ormander constant of $S$. 
\end{definition}

Since $\psi$ equals to $-\infty$ at $S$, \eqref{eq: eM} implies that
$$
|\lambda-\lambda'|\geq 2\epsilon_0, \ \ \forall \ \lambda\neq \lambda' \in S,
$$
thus $S$ is uniformly discrete.  The proof of the H\"ormander criterion also implies:

\begin{theorem}\label{th:h2} Let $S$ be a discrete set in $\mathbb C^n$.  Assume that $\iota(S)>1$. Then there exists a constant $C>0$ such that for every sequence of complex numbers $a=\{a_\lambda\}_{\lambda\in S}$ with 
$$
|a|^2:=\sum_{\lambda\in S} |a_\lambda|^2 e^{-\pi|\lambda|^2}=1,
$$
there exists $F\in \mathcal F^2$ such that $F(\lambda)=a_\lambda$ for all $\lambda\in S$ and $||F||^2 \leq C$.
\end{theorem}

\begin{definition}  Let $S$ be a discrete set in $\mathbb C^n$. We shall define the upper uniform density of $S$ as
$$
D^{+}(S):= \limsup_{r\to\infty} \sup_{z_0\in \mathbb C^n} \frac{n(z_0, r)}{\pi^nr^{2n}/ n!}, 
$$
where $n(z_0, r)$ denotes the number of points in
$$
B(z_0, r):=\{z\in \mathbb C: |z-z_0|<r\}.
$$
\end{definition}

In case $S$ is a lattice in $\mathbb C^n$, we know that $D^{+}(S)^{-1}$ is equal to the Lebesgue measure of the torus $\mathbb C^n/S$. In the one-dimensional case, we also have the following general result.

\begin{theorem} Let $S$ be a uniformly discrete set in $\mathbb C$. Then $D^{+}(S) \cdot \iota(S)=1$
\end{theorem} 

\begin{proof}
Since $S$ is uniformly discrete, we have $\iota(S) >0$, a change of variable argument gives
$$
\iota((\iota(S)-\varepsilon)^{-1/2}S )=(\iota(S)-\varepsilon)^{-1}\epsilon(S) >1
$$
for every sufficiently small $\varepsilon >0$. Thus Theorem \ref{th:h2} implies that $
(\iota(S)-\varepsilon)^{-1/2}S$ is $\mathcal F^2$ interpolating. Apply the main result in \cite{OS}, we know that
$$
1>D^{+}((\iota(S)-\varepsilon)^{-1/2}S) =D^{+}(S)\cdot (\iota(S)-\varepsilon).
$$
Letting $\varepsilon$ go to zero, we have
$$
D^{+}(S) \cdot \iota(S) \leq 1.
$$
Now it suffices to show that
$$
\iota(S) \geq D^{+}(S)^{-1}.
$$
Assume that $D^{+}(S)^{-1}=b$, then for every $0<a<\gamma<b$,
$$
\sup_{z_0\in \mathbb C} \frac{n(z_0, r)}{\pi r^2} \leq \frac1\gamma,
$$
for all sufficiently large $r$, which gives
$$
\int_{B(z_0,r)} i\partial\dbar |z|^2/2=\pi r^2 \geq \gamma\cdot n(z_0, r)
$$
for all $z_0\in \mathbb C$. Apply the Berndtsson--Ortega construction in \cite[page 113--114]{BO}, the above inequality implies that $\iota(S) \geq a$ for every $a<b$. Hence $\iota(S)  \geq b=D^{+}(S)^{-1}$. The proof is complete.
\end{proof}

Using results from \cite{BO, OS}, one may also generalize the above theorem to general weight function $\phi$ with $\phi_{z\bar z}$ bounded by two positive constants. For general higher-dimensional cases, by \cite[Theorem 2]{Lindholm},  we know that if $S$ is $\mathcal F^2$ interpolating then $
D^{+}(S) \leq 1$.
However, in general, $S$ may not be $\mathcal F^2$ interpolating even $D^{+}(S)$ is small enough. Comparing with Theorem \ref{th:h2}, this means that there exists $S$ with very small upper uniform density whose H\"ormander constant is also small.

\subsection{Proof of Theorem \ref{th:uniformity}} Notice that if $((\mathbb Z\oplus \frac12 \mathbb Z)^2, e^{-\pi|t|^2})$ gives a frame in $L^2(\mathbb R^2)$ then $(\mathbb Z^2, e^{-\pi t^2})$ defines a frame in $L^2(\mathbb R)$, which is not true by the Balian-Low type theorem. Now it suffices to prove (2) since (2) implies (1) by the remark after Corollary \ref{co:ff}. Put
$$
X:=\mathbb R^4/(\mathbb Z\oplus2\mathbb Z)^2,
$$
then we know that $(X,\omega)$ is of type $(1,4)$.
Since the moduli space of polarized type $(d_1,d_2)$ ($d_1, d_2$ are fixed positive integers) Abelian surfaces is equal to the Siegel upper half-space,  to prove (2), by the H\"ormander criterion, it suffices to show that the Seshadri constant of a generic polarized type $(1,4)$ Abelian surface is bigger than two. Since generically a polarized type $(1,4)$ Abelian surface has Picard number one, by Theorem 6.1 (b) in \cite{Bauer}, its Seshadri constant equals 
$$
8/\sqrt{8+1}=8/3 >2,
$$
where we use the fact that $k=1, l=3$ is the primitive solution of the following Pell's equation
$$
l^2-8k^2=1.
$$
The proof is complete.

\section{Non transcendental examples}

\subsection{Gr\"ochenig--Lyubarskii's example} Let us look at the following lattice in $\mathbb R^2 \times \mathbb R^2 $ (see \cite[page 3, (4)]{GL}): 
$$
 \Lambda=\left\lbrace \left(e+\frac12 f,\ \frac{\sqrt 3}{2}f\right): e,f \in \mathbb Z^2 \right\rbrace.
$$
Its symplectic dual is
$$
\Lambda^\circ=\left\lbrace \left(e^*,\ \frac{-1}{\sqrt 3}e^*+ \frac{2}{\sqrt 3}f^*\right): e^*,f^* \in \mathbb Z^2 \right\rbrace. 
$$
Fix $\Omega=i I_n$, where $I_n$ denotes the identity matrix. With the notation in Proposition \ref{pr:1.1}, we have
$$
\Gamma_{\Omega, \Lambda^\circ}=\left\lbrace \left(e^*,\ \frac{-1}{\sqrt 3}e^*+ \frac{2}{\sqrt 3}f^*\right): e^*,f^* \in \mathbb Z[i] \right\rbrace.
$$
Let us estimate the Seshadri constant of 
$$
\omega=\omega_{\rm euc}:=\frac i2 \sum_{j=1}^2 dz_j \wedge d\bar z_j
$$
on the complex tori $X:=\mathbb C^2/\Gamma_{\Omega, \Lambda^\circ}$. Notice that the Riemannian metric induced by $\omega$ is  precisely the euclidean metric $|\cdot |$, hence
$$
\inf_{\lambda \neq \lambda' \in \Gamma_{\Omega, \Lambda^\circ}} |\lambda-\lambda'|^2 =\inf_{0 \neq \mu \in \Gamma_{\Omega, \Lambda^\circ}} |\mu|^2  = |(1, -1/\sqrt3)|^2 =\frac43.
$$
Thus the following ball
$$
B:=\{|z|<1/\sqrt 3\}
$$
contains precisely one point in $\Gamma_{\Omega, \Lambda^\circ}$ and we can think of $B$ as a K\"ahler ball in $X$, which gives (see Theorem \ref{th:MP}) the following Seshadri constant inequality
\begin{equation}\label{eq:GL-0}
\epsilon_0(\omega) \geq \frac\pi3 >1.
\end{equation}
However, from \cite[page 3, (4)]{GL}, we know that $(\Lambda, g_\Gamma)$ does not define a frame in $L^2(\mathbb R^2)$. Thus by Proposition \ref{pr:1.1}, $\Gamma_{\Omega, \Lambda^\circ}$ is not a set of interpolation for $\mathcal F^2$. To summarize, we obtain:

\begin{theorem}\label{th:GL-1} There exists a lattice in $\mathbb C^2$ whose Seshadri constant is bigger than one but it is not a set of interpolation for $\mathcal F^2$.
\end{theorem} 

\noindent
\textbf{Remark.}  \emph{By the H\"ormander criterion, we know every lattice in $\mathbb C^2$ with Seshadri constant bigger than two is a set of interpolation for $\mathcal F^2$.  In the above example, one may further prove that
\begin{equation}\label{eq:GL}
\frac 43\geq \epsilon_0(\omega)  \geq \frac\pi3.
\end{equation}
In fact the complex line $\mathbb C(0, 2/\sqrt 3)$ covers a subtorus, say 
$$
V\simeq \mathbb C / (2/\sqrt 3)\mathbb Z[i],
$$  
of $X$, thus \cite[Theorem 4.6]{Tosatti} implies that
$$
 (2/\sqrt 3)^2=\int_V \omega \geq \epsilon_0(\omega),
$$
from which \eqref{eq:GL} follows.}

\subsection{Complex lattices}

We call $\Gamma$ a \emph{complex lattice} if 
$$
i \Gamma= \Gamma.
$$
One may verify the following:

\begin{proposition} For a lattice $\Gamma$ in $\mathbb C^n$, the followings are equivalent
\begin{itemize}
\item[(1)] $\Gamma$ is a complex lattice;
\item[(2)] $\Gamma=\mathbb Z[i] \{\gamma_1,\cdots, \gamma_n\}$ for some $\gamma_j\in\mathbb C^n$;
\item[(3)] $\Gamma=A \mathbb Z[i]^n$ for some $A\in GL(n, \mathbb C)$;
\item[(4)] $X:=\mathbb C^n/\Gamma$ is biholomorphic to $\mathbb C^n/\mathbb Z[i]^n$.
\end{itemize}
\end{proposition}

Now we can prove a generalization of \eqref{eq:GL}.

\begin{theorem}\label{th:gl} Assume that $\Gamma=A \mathbb Z[i]^n$ is a complex lattice. Then 
$$
m(\Gamma) \geq \epsilon_0(\omega)\geq \max\left\lbrace\frac\pi 4 m(\Gamma), \ e_{min}(A)\right\rbrace,
$$
where 
$$
e_{min}(A):=\inf_{z\in \mathbb C^n, \, |z|=1} |Az|^2.
$$
\end{theorem}

\begin{proof}  Choose $z=A^{-1}w$ as the new variable, one may assume that $A=I_n$. Then
$$
\omega=\frac i2 \sum_{j=1}^n dw_j \wedge d\bar w_j \geq  e_{min}(A) \cdot \omega_0, \ \ \ \omega_0:=\frac i2 \sum_{j=1}^n dz_j \wedge d\bar z_j.
$$
Denote by $\psi_0(z)$ the Green function on $\mathbb C/\mathbb Z[i]$ satisfying
$$
i\pi  dz \wedge d\bar z +i\partial\dbar\psi_0= i\partial\dbar \log|z|^2.
$$
then we know that
$$
\psi:= e_{min}(A) \cdot \max\{\psi_0(z_1),\cdots, \psi_0(z_n)\}
$$
satisfies  $2\pi \omega +i\partial\dbar\psi \geq 0$ and has order one log pole at $\Gamma$. Thus we know that $
\epsilon_0(\omega)\geq e_{min}(A)$, together with Theorem \ref{th:MP}, it gives the lower bound of the Seshadri constant that we need. To prove the upper bound, it suffices to choose a subtorus
$$
V_\gamma:=\mathbb C/ \mathbb Z[i]\gamma,  \ \  0\neq \gamma\in \Gamma,
$$
then \cite[Theorem 4.6]{Tosatti} gives
$$
\epsilon_0(\omega) \leq \int_{V_\gamma} \omega=|\gamma|^2. 
$$
Take the infimum over all $0\neq \gamma\in \Gamma$, the upper bound follows.
\end{proof}

\noindent
\textbf{Remark.}  \emph{In case 
$$
\Gamma= a_1 \mathbb Z[i] \times \cdots \times a_n \mathbb Z[i], \ \ a_j >0,
$$
we have
$$
e_{min}(A)=m(\Gamma) = \min\{a^2_1,\cdots, a^2_n\},
$$
thus the above theorem gives
$$
\epsilon_0(\omega)=\min\{a^2_1,\cdots, a^2_n\}.
$$}

\subsection{Seshadri sequence and Gr\"ochenig's result}

In this section, we shall rephrase the main result of Gr\"ochenig in \cite{Gr11} in terms of the Seshadri constant. The main idea is to consider a sequence of extensions, more precisely, let
\begin{equation}\label{eq:sequence}
\{0\}=X_0\subset X_1\cdots \subset X_k=X:=\mathbb C^n/\Gamma, \ \ n_k:=\dim_{\mathbb C} X_k, \ \  k\geq 1,
\end{equation} 
be an increasing sequence of complex Lie subgroups of $X$. We shall introduce the Seshadri constant $\epsilon_j$, $1\leq j\leq k$, for extension from $X_{j-1}$ to $X_{j}$. Let
$$
\pi_j: E_j \to X_j
$$
be the covering map, where $E_j$ is an $n_j$ dimensional complex subspace of $\mathbb C^n$.  Let
$$
E_j= E_{j-1}\oplus F_j,
$$ 
be the orthogonal decomposition with respect to the Euclidean metric $\omega$.  Then
$$
\Gamma_j:=F_j \cap \pi_j^{-1}(X_{j-1})
$$
define a lattice in $F_j$. Put
$$
X^\bot_{j-1}:=F_j/ \Gamma_j,
$$
(in general, $X^\bot_{j-1}$ is not a subtorus of $X_j$). Denote by 
\begin{equation}\label{eq:ej}
\epsilon_j:= \epsilon_0(\omega; X^\bot_{j-1})
\end{equation}
the Seshadri constant at the origin of $X^\bot_{j-1}$ with respect to $\omega$.

\begin{definition}\label{de:shsha-4}  We call \eqref{eq:sequence}  an admissible sequence of $X$ if 
$$
\epsilon_j > n_j-n_{j-1}, \ \ \forall \ 1\leq j\leq n.
$$
$X$ is said to be Seshadri admissible if it possesses an admissible sequence. 
\end{definition}

\begin{theorem}\label{th:sa}
Assume that  $X$ is Seshadri admissible, then $\Gamma$ is a set of interpolation in $\mathcal F^2$.
\end{theorem}

\begin{proof}
Let $\pi: \mathbb C^n \to X$ be the covering map. Put
$$
\mathcal F^2_j:=\left\lbrace F \ \text{holomorphic on}\ \pi^{-1}(X_j): \int_{\pi^{-1}(X_j)} |F|^2 e^{-\pi|z|^2} <\infty \right\rbrace.
$$
It suffices to prove that each element $f$ in $\mathcal F^2_{j-1}$ extends to an element $F$ in $\mathcal F^2_j$ with $||F||\leq C_j ||f||$. Since $\pi^{-1}(X_j)$ is a disjoint union of translates of $E_j$ and
$$
\pi^{-1}(X_{j-1}) \cap E_j=\pi_j^{-1}(X_{j-1}),
$$
the above extension problem reduces to the extension from $\mathcal F^2_{j-1}|_{\pi_j^{-1}(X_{j-1})}$ to $\mathcal F^2_j|_{E_j}$. Apply the H\"ormander method, it suffices to construct an $\omega$ plurisubharmonic function with order $(n_j-n_{j-1})$ log pole along $\pi_j^{-1}(X_{j-1})$.  Now the assumption $\epsilon_j > n_j-n_{j-1}$ gives an $\omega$ plurisubharmonic $\psi_j$ with  order $n_j-n_{j-1}$ log pole at the origin of  $F_j$, the pull back of $\psi_j$ along the natural projection
$$
E_j \to F_j
$$ 
gives the function that we need.
\end{proof}

Now we shall show how to use the above theorem to give a new proof of \cite[Theorem 9]{Gr11} on the Gabor frame property for $(\Lambda, g_\Omega)$.  The setup  for \cite[Theorem 9]{Gr11} is the following:
$$
\Omega=iI_n, \ \ \Gamma_{\Omega, \Lambda^\circ} \ \text{is a complex lattice}.
$$
Based on Proposition \ref{pr:1.1}, we shall prove a similar result with a weaker assumption, i.e. we shall only assume that $ \Gamma_{\Omega, \Lambda^\circ}$ is a complex lattice. Let us write
$$
\Gamma_{\Omega, \Lambda^\circ}= A \mathbb Z[i]^n,
$$
where $A\in GL(n, \mathbb C)$. By the Iwasawa decomposition (see \cite[Proposition 26.1]{Bump}), we have
$$
A=US,
$$
where $U$ is unitary and $S$ is lower triangular with positive eigenvalues $\lambda_j$ ($U, S$ are uniquely determined by $A$, $\lambda_j^{-1}$ is equal to $\gamma_j$ in \cite[Theorem 9]{Gr11}).
Since the Euclidean metric $\omega$ is unitary invariant, one may assume that
$$
\Gamma_{\Omega, \Lambda^\circ}=S \mathbb Z[i]^n.
$$
Put
$$
E_j=\{z\in \mathbb C^n: z_1=\cdots=z_{n-j}=0\},  \ \ 1\leq j\leq n,
$$
and
$$
X_j=\pi(E_j).
$$
Then we have
$$
F_j \simeq \mathbb C, \  \ \  \Gamma_j\simeq \lambda_{n-j+1} \mathbb Z[i],
$$
and
$$
\epsilon_j=\lambda_{n-j+1}^2, \ \  1\leq j\leq n.
$$
Since $n_j=j$, we have $n_j-n_{j-1}=1$. Hence if
$$
\lambda_{n-j+1} >1, \ \ 1\leq j\leq n,
$$ 
then $X$ is Seshadri admissible. Thus Theorem \ref{th:sa} implies the following slight generalization of Gr\"ochenig's result (notice again that 
$\gamma_j =\lambda_j^{-1}$):

\begin{theorem}[Theorem 9 in \cite{Gr11}] Let $\Gamma_{\Omega, \Lambda^\circ}$ be a complex lattice. With the notation above, assume that $\lambda_j>1$ for all $1\leq j\leq n$. Then $\Gamma_{\Omega, \Lambda^\circ}$ is set of  interpolation for $\mathcal  F^2$ (and equivalently $(\Lambda, g_\Omega)$ defines a frame in $L^2(\mathbb R^n)$).
\end{theorem}

\section{Effective interpolation bounds}

\subsection{Interpolation bounds in terms of the Buser--Sarnak constant}

We shall use Theorem \ref{th:OT-xu} to prove the following:

\begin{theorem}\label{th:c1} Fix a lattice $\Gamma$ in $\mathbb C^n$. 
If \begin{equation}\label{eq:Bs-1}
C= \frac\pi4\cdot \inf_{0 \neq \mu \in\Gamma}|\mu|^2  > n
\end{equation}
then every sequence of complex numbers $a=\{a_\gamma\}$ with $\sum_{\gamma\in \Gamma} |a_\gamma|^2 e^{-\pi|\gamma|^2}=1$ extends to a function in $\mathcal F^2$. Moreover, 
$$
1-e^{-C} \sum_{k=0}^{n-1} \frac{C^k}{k!}\leq \inf_{F\in \mathcal F^2, \ F(\gamma)=a_\gamma, \ \forall 
\ \gamma\in\Gamma} ||F||^2  \leq \frac{M(C)}{ n! \,e^n },
$$
where
$$
M(C):=
\begin{cases}
(n+1)^{n+1} & \text{if}   \  \ C\geq n+1\\ 
C^{n+1}/(C-n) & \text{if}   \  \ n<C< n+1.
\end{cases}
$$
\end{theorem}

\begin{proof}[Proof of the lower bound] Notice that (by an induction on $n$)
$$
1-e^{-C} \sum_{k=0}^{n-1} \frac{C^k}{k!}= \int_{0}^C e^{-t} \frac{t^{n-1}}{(n-1)!}\, dt
$$
and 
$$
\{z\in \mathbb C^n: \pi |z-\gamma|^2<C\} \cap \{z\in \mathbb C^n: \pi |z-\gamma'|^2<C\}=\emptyset, \ \ \forall \ \gamma\neq \gamma' \in \Gamma,
$$
It suffices to show that
\begin{equation}\label{eq:lower-1}
\int_{\{\pi |z-\gamma|^2<C\}} |F(z)|^2e^{-\pi|z|^2} \geq |F(\gamma)|^2 e^{-\pi|\gamma|^2}\int_{0}^C e^{-t} \frac{t^{n-1}}{(n-1)!}\, dt.
\end{equation}
Notice that
$$
\int_{\{\pi |z-\gamma|^2<C\}} |F(z)|^2e^{-\pi|z|^2}=e^{-\pi|\gamma|^2}\int_{\{\pi |w|^2<C\}} |F(w+\gamma)e^{-\pi w^T\gamma}|^2e^{-\pi|w|^2},
$$
The main observation is that the Taylor expansion
$$
G(w):=F(w+\gamma)e^{-\pi w^T\gamma}= F(\gamma)+\sum c_\alpha w^\alpha,
$$
is now an orthogonal decomposition, i.e.
\begin{align*}
\int_{\{\pi |w|^2<C\}} |G(w)|^2e^{-\pi|w|^2} & =|F(\gamma)|^2 \int_{\{\pi |w|^2<C\}} e^{-\pi|w|^2} \\
& +\sum |c_\alpha|^2 \int_{\{\pi |w|^2<C\}} |w^\alpha|^2e^{-\pi|w|^2},
\end{align*}
from which we know that
$$
\int_{\{\pi |z-\gamma|^2<C\}} |F(z)|^2e^{-\pi|z|^2} \geq |F(\gamma)|^2 e^{-\pi|\gamma|^2}\int_{\{\pi |w|^2<C\}} e^{-\pi|w|^2}.
$$
Now put $t=\pi|\gamma|^2$, we have
$$
\int_{\{\pi |w|^2<C\}} e^{-\pi|w|^2}= \int_0^C e^{-t}d\, \frac{t^n}{n!} =\int_{0}^C e^{-t} \frac{t^{n-1}}{(n-1)!}\, dt,
$$
which gives \eqref{eq:lower-1}.
\end{proof}

\begin{proof}[Proof of the upper bound] The main idea is to use Theorem \ref{th:OT-xu}. The definition of $C$ implies that
$$
B:=\{z\in \mathbb C^n: \pi |z|^2<C\}
$$
is embedded ball in $X:=\mathbb C^n /\Gamma$. For 
$$
\frac\pi C\leq  \delta < \frac\pi n,
$$
put (notice that $\{\delta|z|^2 <1\} \subset B$)
$$
\psi_\delta(z):=
\begin{cases}
\log(\delta|z|^2)+(1-\delta|z|^2)  & \  \ \delta|z|^2 <1 \\
0 &    z\in X\setminus \{\delta|z|^2 <1\}.
\end{cases}
$$
Then we have
$$
i\partial\dbar \psi_\delta \geq  -\delta\cdot i\partial\dbar |z|^2.
$$
Denote by $\psi$ the pull back to $\mathbb C^n$ of $ \pi \psi_\delta/\delta$. Apply Theorem \ref{th:OT-xu} to $\psi$, we get
$$
\inf_{F\in \mathcal F^2, \ F(\gamma)=a_\gamma, \ \forall 
\ \gamma\in\Gamma} ||F||^2  \leq  \left(1-\frac{n\delta}{\pi}\right)^{-1} \frac{\pi^n}{n!}  e^{-n(1+\log \delta)}.
$$
Put
$$
x:=\frac{n\delta}{\pi} ,
$$
then $n/C\leq x<1$ and
$$
\left(1-\frac{n\delta}{\pi}\right)^{-1} \frac{\pi^n}{n!}  e^{-n(1+\log \delta)} =\frac{n^n}{(1-x)x^n} \cdot \frac1{n!\,e^n}.
$$
Thus the upper bound follows from
$$
\inf_{n/C\leq x<1}  \frac{n^n}{(1-x)x^n}  =M(C).
$$
The proof of Theorem B is now  complete.
\end{proof}

\subsection{Interpolation bounds in terms of the Robin constant}

In this subsection, we shall generalize Theorem \ref{th:c1} to the case that $\epsilon_0(\omega)>n$.  The main idea is to consider the following \emph{envelope with prescribed singularity}
\begin{align*}
\psi_a^\pi:=  \sup \{\psi_0\leq 0: \ \ & 2\pi\omega+i\partial\dbar \psi_0 \geq 0, \ \psi_0  \ \text{has an isolated order} \ a \\
& \text{log pole at the origin} \}
\end{align*}
on the torus $X=\mathbb C^n/\Gamma$. Denote by $\psi_a$ the pull back to $\mathbb C^n$ of $\psi_a^\pi$. 

\begin{definition}\label{de:Robin} We call $\psi_a$ the $a$-envelope function on $\mathbb C^n$ associated to the lattice $\Gamma$ and
$$
\rho_a:=\liminf_{z\to 0} \psi_a(z)-a \log|z|^2.
$$
the $a$-Robin constant of $\Gamma$.
\end{definition}

We have the following generalization of Theorem \ref{th:c1} and Theorem B.

\begin{theorem}\label{th:c2}Assume that $\epsilon:=\epsilon_0(\omega)>n$. Put\begin{equation}\label{eq:Bs-2}
C= \frac\pi4\cdot \inf_{0 \neq \mu \in\Gamma}|\mu|^2.
\end{equation}
Then every sequence of complex numbers $c=\{c_\gamma\}$ with $\sum_{\gamma\in \Gamma} |c_\gamma|^2 e^{-\pi|\gamma|^2}=1$ extends to a function in $\mathcal F^2$. Moreover, 
$$
1-e^{-C} \sum_{k=0}^{n-1} \frac{C^k}{k!}\leq \inf_{F\in \mathcal F^2, \ F(\gamma)=c_\gamma, \ \forall 
\ \gamma\in\Gamma} ||F||^2  \leq \inf_{n<a<\epsilon}\left(1-\frac na\right)^{-1} \frac{\pi^n}{n!}\,  e^{-n\rho_a/a}.
$$
\end{theorem}

\begin{proof} The proof of the lower bound is the same. For the upper bound, it suffices to apply Theorem \ref{th:OT-xu} to $\psi=\psi_a$.
\end{proof}

\noindent
\textbf{Remark.}  \emph{In case $n=1$, we know that
$$
\epsilon=\int_{\mathbb C/\Gamma} \omega.
$$
Moreover, $\psi_\epsilon^\pi$ is also well defined. In fact, we have the following}

\begin{proposition} $\psi_\epsilon^\pi$ is equal to the unique solution, say $\psi$,  of 
$$
i\partial\dbar \psi +2\pi \omega=\epsilon\cdot i\partial\dbar \log|z|^2, \ \  \ \ \sup_{X} \psi=0,
$$
on $X=\mathbb C/\Gamma$.
\end{proposition}

\begin{proof} Since
$$
\int_X i\partial\dbar \psi +2\pi \omega = \int_X 2\pi \omega =2\pi \epsilon= \int_X\epsilon\cdot i\partial\dbar \log|z|^2,
$$
from the Hodge theory, we know that up to a constant there exists a unique solution $\psi$ such that
$$
i\partial\dbar \psi +2\pi \omega=\epsilon\cdot i\partial\dbar \log|z|^2.
$$
Thus if we assume further that $\sup_X \psi=0$ then $\psi$ is unique. Moreover, we know that $\psi$ is smooth outside the origin and $\psi-\epsilon \log|z|^2$ is smooth near the origin, thus 
$$
\psi \leq \psi_\epsilon^\pi.
$$ 
On the other hand, notice that $\psi_\epsilon^\pi -\psi$ is subharmonic, thus $\psi_\epsilon^\pi -\psi$ is equal to a constant, say $A$.  Take $z_0$ such that $\psi(z_0)=0$, then
$$
A=\psi_\epsilon^\pi(z_0) -\psi(z_0) =\psi_\epsilon^\pi(z_0) \leq 0,
$$
which gives
$$
\psi \geq \psi_\epsilon^\pi.
$$
Hence $\psi = \psi_\epsilon^\pi$.
\end{proof}

\noindent
\textbf{Remark.}  \emph{The above Proposition implies that $\psi_\epsilon^\pi$ is equal to the \emph{Arakelov Green function} up to a \emph{non-zero} constant (see \cite[page 393 and 417]{Faltings}).}

\subsection{Faltings' identity for the Robin constant}\label{se:5.3}

Notice that
$$
\int_X \omega/\epsilon=1,
$$
we know that $[\omega/\epsilon]$ is the Chern class of the line bundle, $L=[0]$, over $X$. Choose a metric $h$ on $L$ such that the Chern curvature satisfies
$$
i\Theta(L, h)=2\pi \omega/\epsilon.
$$
Denote by $s$ the canonical section of $[0]$, then we know that
$$
i\partial\dbar \log|s|^2_h= i\partial\dbar \log|z|^2- 2\pi \omega/\epsilon,
$$
which implies that
\begin{equation}\label{eq:Green}
\epsilon \log|s|^2_h-\sup_X \epsilon \log|s|^2_h= \psi_\epsilon^\pi.
\end{equation}
The pull back of the line bundle $[0]$ to $\mathbb C$ is a trivial line bundle, thus one may identify the pull back  to $\mathbb C$ of $s$ with a holomorphic function, say $f^s(z)$ on $\mathbb C$. In case $\Gamma={\rm Span}_{\mathbb Z} \{1, \tau\}$, ${\rm Im}\, \tau=\epsilon$, (up to a constant) we have
$$
f^s(z)=\vartheta\left(z+\frac12+\frac \tau2; \tau\right),
$$ 
(notice that $\vartheta\left(\frac12+\frac \tau2; \tau\right)=0$) where
$$
\vartheta(z; \tau)=:\sum_{n\in\mathbb Z} e^{2\pi i nz} e^{\pi i n^2 \tau}
$$
is known as the \emph{Jacobi theta function}. The following formula for the Robin constant is based on the Faltings' \emph{theta metric} $||\theta||$ (see \cite[page 403, 413 and 416]{Faltings} or the function $U$ below).

\begin{proposition}\label{pr:faltings} Assume further that $\Gamma={\rm Span}_{\mathbb Z} \{1, \tau\}$, ${\rm Im}\, \tau=\epsilon$. Denote by $\psi_\epsilon$ the pull back to $\mathbb C$ of $\psi_\epsilon^\pi$, then 
\begin{equation}\label{eq:fal-1}
\frac{\psi_\epsilon}\epsilon= \log U- \sup_{\mathbb C} \log U, \ \ \ U(z):=  \big|\vartheta\left(z+\frac12+\frac \tau2; \tau\right)e^{-\pi ({\rm Im} \, z +\frac\epsilon2)^2/\epsilon}\big|^2,
\end{equation}
and the $\epsilon$-Robin constant of $\Lambda$ defined by 
$$
\rho_\epsilon:=\liminf_{z\to 0} \psi_\epsilon(z)-\epsilon \log|z|^2
$$
satisfies
\begin{equation}\label{eq:fal-2}
\frac{\rho_\epsilon}\epsilon=2 \log(2\pi)+6\log|\eta(\tau)|- \sup_{\mathbb C} \log U,
\end{equation}
where $\eta(\tau)$ is the Dedekind eta function defined by $
\eta(\tau):=e^{\pi i\tau/12} \Pi_{n=1}^\infty (1-e^{2\pi i n \tau})$.
\end{proposition}

\begin{proof} One may verify that $U$ is $\Gamma$-invariant and
$$
i\partial\dbar \log U= -2\pi \omega/\epsilon+\sum_{\gamma\in \Gamma}   i\partial\dbar \log|z+\gamma|^2,
$$ 
thus \eqref{eq:fal-1} follows. To prove \eqref{eq:fal-2}, it suffices to show that
$$
\lim_{z\to 0} U(z)/|z|^2 = (2\pi)^2 \cdot |\eta(\tau)|^6
$$
or equivalently
$$
\big|\frac{\partial \vartheta}{\partial z} \left( \frac12+\frac \tau2; \tau \right) e^{-\pi\epsilon/4} \big|^2 =(2\pi)^2 \cdot |\eta(\tau)|^6,
$$
which follows from
$$
\left(\frac{\partial \vartheta}{\partial z} \left( \frac12+\frac \tau2; \tau \right) e^{\pi i \tau/4}\right)^8=(2\pi)^8 \cdot \eta(\tau)^{24}
$$
(since both sides are cusp forms of degree 12 with the same leading term).
\end{proof}

\noindent
\textbf{Remark.}  \emph{Notice that 
$$
\sup_{\mathbb C} \log U =2\cdot \sup_{z\in \mathbb C}   \log \big|\vartheta\left(z; \tau\right)e^{-\pi ({\rm Im} \, z)^2/\epsilon}\big|
$$ 
and
$$
\sup_{z\in \mathbb C}   \log \big|\vartheta\left(z; \tau\right)e^{-\pi ({\rm Im} \, z)^2/\epsilon}\big| =\sup_{t\in\mathbb R}
\phi(t)- \pi t^2/\epsilon, \ \ \ \phi(t):=\sup_{{\rm Im}\, z=t}  \log |\vartheta(z; \tau)|.
$$
Since $\vartheta$ is an even function of $z$ and depends only on $e^{2\pi iz}$, we know that
$$
\phi(t)=\phi(-t)
$$
is an even convex function of $t$. Moreover, since $|\vartheta\left(z; \tau\right)e^{-\pi ({\rm Im} \, z)^2/\epsilon}\big|$ is $\Gamma$ invariant, we have
$$
\sup_{z\in \mathbb C}   \log \big|\vartheta\left(z; \tau\right)e^{-\pi ({\rm Im} \, z)^2/\epsilon}\big| =\sup_{0<t<\epsilon}
\phi(t)- \pi t^2/\epsilon
$$
and
$$
\phi(t)=\sup_{0<x<1} \log|\sum_{n\in\mathbb Z} e^{2\pi i n(x+it)} e^{\pi i n^2 \tau}|.
$$
Notice that
$$
e^{\phi(t)} \leq \sum_{n\in\mathbb Z} e^{-2\pi nt} e^{-\pi n^2\epsilon}
$$
with identity holds if $\tau=i\epsilon$.
By the Poisson summation formula, we have
$$
\sqrt{\epsilon}\cdot  e^{\pi  z^2/\epsilon}\sum_{n\in\mathbb Z} e^{2\pi i nz} e^{-\pi n^2 \epsilon}= \sum_{n\in\mathbb Z} e^{2\pi i nz/\epsilon} e^{-\pi n^2 /\epsilon},
$$
take $z=it$, we get
$$
e^{\phi(t)-\pi t^2 /\epsilon} \leq e^{-\pi t^2 /\epsilon}\sum_{n\in\mathbb Z} e^{-2\pi nt} e^{-\pi n^2\epsilon}=\frac1{\sqrt{\epsilon}} \cdot\sum_{n\in\mathbb Z} e^{-\pi   nt/\epsilon} e^{-\pi n^2 /\epsilon},
$$
which gives
$$
\sup_{z\in \mathbb C}    \big|\vartheta\left(z; \tau\right)e^{-\pi ({\rm Im} \, z)^2/\epsilon}\big|  \leq\frac1{\sqrt{\epsilon}} \cdot\sum_{n\in\mathbb Z}  e^{-\pi n^2 /\epsilon}=\sum_{n\in\mathbb Z} e^{-\pi n^2 \epsilon},
$$
(for a higher-dimensional generalization of the above argument, see \cite[Lemma 8.2]{PA}). By \eqref{eq:fal-2}, the above estimate gives the following lower bound for $\rho_\epsilon/\epsilon$.}

\begin{theorem}\label{th:robin}  The $\epsilon$-Robin constant of $\Gamma={\rm Span}_{\mathbb Z} \{1, \tau\}$, ${\rm Im}\, \tau=\epsilon$ satisfies
\begin{equation}\label{eq:fal-3}
\frac{\rho_\epsilon}\epsilon \geq 2 \log(2\pi)+6\log|\eta(\tau)|- 2 \log  \sum_{n\in\mathbb Z} e^{-\pi n^2 \epsilon},
\end{equation}
with identity holds  if $\tau= i\epsilon$.
\end{theorem}

\subsection{Proof of Theorem B}

\begin{proof}[Proof of (1) and (2)] Notice that the disc of diameter $\int_{0\neq \lambda \in \Gamma} |\lambda|$ is contained in a fundamental domain of $\mathbb C/\Gamma$. Hence $|\Lambda|^{-1}=|\Gamma| \geq C$. The frame bounds estimate follows directly from Theorem \ref{th:c1} and Theorem \ref{th:dual3}.
\end{proof}

\begin{proof}[Proof of (3)] For the lower bound, by Theorem \ref{th:c2} (let $a$ go to $\epsilon$), it suffices to compute the $\epsilon$-Robin constant $\rho$, where  
$$
\epsilon:=a \, {\rm Im}\, \tau
$$
is the Seshadri constant. Denote by $\rho'$ the ${\rm Im}\, \tau/a$-Robin constant  of $\mathbb C/\langle 1, \tau/a\rangle$, then a change of variable argument gives
$$
\rho=a^2 \rho'.
$$  
Thus 
$$
\rho/\epsilon=\rho'/( {\rm Im}\, \tau/a) \geq 2 \log(2\pi)
+6 \log|\eta(\tau/a)|-2 \log  \sum_{n\in\mathbb Z} e^{-\pi n^2  {\rm Im}\, \tau/a},
$$
by Theorem \ref{th:robin}. Apply Theorem \ref{th:c2} and Theorem \ref{th:dual3}, we get the lower bound. The upper bound follows directly from Theorem \ref{th:dual3} and Theorem \ref{th:c1}. 
\end{proof}
\textbf{Acknowledgement:} We would like to thank A. Austad, E. Berge, U. Enstad, M. Faulhuber, L. Polterovich and E. Skrettingland for their feedback on earlier versions of the manuscript. Thanks are due to the referee for many helpful suggestions.

\end{document}